\documentclass[12pt]{article}
\usepackage{amssymb,amsfonts,epic
}
\begin{document}

\newtheorem{lem}{Lemma}[section]
\newtheorem{prop}{Proposition}
\newtheorem{con}{Construction}[section]
\newtheorem{defi}{Definition}[section]
\newtheorem{coro}{Corollary}[section]
\newcommand{\hf}{\hat{f}}
\newtheorem{fact}{Fact}[section]
\newtheorem{theo}{Theorem}
\newcommand{\Br}{\Poin}
\newcommand{\Cr}{{\bf Cr}}
\newcommand{\dist}{{\bf dist}}
\newcommand{\diam}{\mbox{diam}\, }
\newcommand{\mod}{{\rm mod}\,}
\newcommand{\compose}{\circ}
\newcommand{\dbar}{\bar{\partial}}
\newcommand{\Def}[1]{{{\em #1}}}
\newcommand{\dx}[1]{\frac{\partial #1}{\partial x}}
\newcommand{\dy}[1]{\frac{\partial #1}{\partial y}}
\newcommand{\Res}[2]{{#1}\raisebox{-.4ex}{$\left|\,_{#2}\right.$}}
\newcommand{\sgn}{{\rm sgn}}

\newcommand{\C}{{\bf C}}
\newcommand{\D}{{\bf D}}
\newcommand{\Dm}{{\bf D_-}}
\newcommand{\N}{{\bf N}}
\newcommand{\R}{{\bf R}}
\newcommand{\Z}{{\bf Z}}
\newcommand{\tr}{\mbox{Tr}\,}

\newenvironment{nproof}[1]{\trivlist\item[\hskip \labelsep{\bf Proof{#1}.}]}
{\begin{flushright} $\square$\end{flushright}\endtrivlist}
\newenvironment{proof}{\begin{nproof}{}}{\end{nproof}}

\newenvironment{block}[1]{\trivlist\item[\hskip \labelsep{{#1}.}]}{\endtrivlist}
\newenvironment{definition}{\begin{block}{\bf Definition}}{\end{block}}

\newtheorem{conjec}{Conjecture}
\newtheorem{com}{Comment}
\newtheorem{exa}{Example}
\font\mathfonta=msam10 at 11pt
\font\mathfontb=msbm10 at 11pt
\def\Bbb#1{\mbox{\mathfontb #1}}
\def\lesssim{\mbox{\mathfonta.}}
\def\suppset{\mbox{\mathfonta{c}}}
\def\subbset{\mbox{\mathfonta{b}}}
\def\grtsim{\mbox{\mathfonta\&}}
\def\gtrsim{\mbox{\mathfonta\&}}

\newcommand{\ar}{{\bf area}}
\newcommand{\1}{{\bf 1}}
\newcommand{\Bo}{\Box^{n}_{i}}
\newcommand{\Di}{{\cal D}}
\newcommand{\gd}{{\underline \gamma}}
\newcommand{\gu}{{\underline g }}
\newcommand{\ce}{\mbox{III}}
\newcommand{\be}{\mbox{II}}
\newcommand{\F}{\cal{F_{\eta}}}
\newcommand{\Ci}{\bf{C}}
\newcommand{\ai}{\mbox{I}}
\newcommand{\dupap}{\partial^{+}}
\newcommand{\dm}{\partial^{-}}
\newenvironment{note}{\begin{sc}{\bf Note}}{\end{sc}}
\newenvironment{notes}{\begin{sc}{\bf Notes}\ \par\begin{enumerate}}%
{\end{enumerate}\end{sc}}
\newenvironment{sol}
{{\bf Solution:}\newline}{\begin{flushright}
{\bf QED}\end{flushright}}

\title{Multipliers of periodic orbits in spaces of rational maps
\date{}
}
\author{Genadi Levin
\thanks{Supported in part by EU FP6 Marie Curie ToK programme SPADE2 at IMPAN, Poland}\\
\small{Institute of Mathematics, the Hebrew University of Jerusalem, Israel}\\
} \normalsize \maketitle

\centerline{\it \ \ \ \ \ \ \ \ \ \ \ \ \ \ \ \ \ \ \ \ \ \ \ \ \ \ \ \ \ \ \
\ \ \ \ \ \ \ \ Dedicated to the memory of Adrien Douady}

\

\abstract{Given a polynomial or
a rational function $f$ we include it in a space of maps. 
We introduce local coordinates in this space, which are essentially the
set of critical values of the map. Then we consider an arbitrary
periodic orbit of $f$ with multiplier $\rho\not=1$ as a function
of the local coordinates, and establish a simple connection
between the dynamical plane of $f$ and the function $\rho$ 
in the space associated to $f$. 
The proof is based on the theory of
quasiconformal deformations of
rational maps. As a corollary, we show that multipliers of
non-repelling periodic orbits are also local coordinates in the
space.
} \

\section{Introduction}
The multiplier map of an attracting periodic orbit of 
a quadratic polynomial within the quadratic family
$z\mapsto z^2+v$ uniformizes the component of parameters $v$, for which
it is attracting. This theorem~\cite{DH1},~\cite{mcmsul},~\cite{CG}
is a cornerstone of the Douady-Hubbard theory of the Mandelbrot set.
It has been generalized to
components of hyperbolic polynomials~\cite{mil} and degree two
rational maps~\cite{mary}.
 
In this paper we develop a local approach
to the problem of behavior of multipliers of periodic
orbits in general spaces of polynomials and rational maps.
Our approach is somewhat closer to~\cite{Eps} and especially~\cite{Leij0}.
As a corollary, we show that
the multiplier maps of attracting and neutral periodic orbits
are local coordinates in the (moduli) spaces of polynomials and rational maps.
It includes in particular the cited above Douady-Hubbard-Sullivan
theorem.

Let us state our main results. Let $f$ be a rational function of
degree $d\ge 2$. We consider a {\it Ruelle transfer operator}
$T=T_f$, which acts on functions $\psi$
as follows:
$$T\psi(z)=\sum_{w:f(w)=z}\frac{\psi(w)}{(f'(w))^2},$$
provided $z$ is not a critical value of $f$. Let us also consider 
any {\it periodic orbit} $O=\{b_k\}_{k=1}^n$ of $f$ of exact
period $n$. Denote its {\it multiplier} by $\rho$:
$$\rho=(f^n)'(b_k)=\Pi_{j=1}^n f'(b_j).$$
We assume that $\rho\not=1$. 
(For $\rho=1$, see Sects.~\ref{cusp},~\ref{rcusp}.)
Let us associate to the
periodic orbit $O$ of $f$ a rational function $B_O$:
\begin{equation}\label{iB-func}
B_O(z)=\sum_{k=1}^n \frac{\rho}{(z-b_k)^2}+\frac{1}{1-\rho}
\sum_{k=1}^n \frac{(f^n)''(b_k)}{z-b_k}.
\end{equation}
It was
introduced in~\cite{Leij0} for the unicritical family.

On the other hand, we include $f$ in a natural space of rational
maps of the same degree. Roughly speaking, this is the set of maps
with fixed multiplicities at different critical points and similar
behavior at $\infty$. 
For instance, the space associated to
a unicritical polynomial $z^d+v_0$ is the unicritical
family $z^d+v$, $v\in {\bf C}$. 
We introduce local coordinates in the space
near $f$, which are basically the set of critical values
$v_1,...,v_p$. When $f$, hence, the periodic orbit,
moves in the space, $\rho$ becomes a holomorphic function in these
coordinates. Our main aim is to show that the following
connection holds between the dynamical and the parameter spaces of $f$:
\begin{equation}\label{iruelle}
B_O(z)-(TB_O)(z)= \sum_{j=1}^{p}\frac{\partial \rho}{\partial
v_j}\frac{1}{z-v_j}.
\end{equation}
For unicritical polynomials, this connection appeared
and was proved in~\cite{Leij0}. It has been applied to the 
problems of geometry of Julia and the
Mandelbrot sets in~\cite{Leij0},~\cite{leij1},~\cite{le2}.
\begin{com}\label{ae}
Considered on
quadratic differentials $\psi(z) dz^2$ the operator $T$ is 
a so-called
pushforward operator $\sum_{f(w)=z} \psi(w) dw^2$ introduced to the field
probably by Thurston in his work 
on critically finite branched covering maps of the sphere, see~\cite{dht}.
In completely different applications it appeared in
\cite{lman},~\cite{lpade},~\cite{sy}. 
Explicit formulae for the Fredholm determinant of $T$
in spaces of analytic functions
have been proved and used in~\cite{LSY},~\cite{LSY1},~\cite{Le1},~\cite{ELS}.
See also~\cite{Ts}. 
The action of $T$ on quadratic differentials with multiple poles
was first studied and used in~\cite{Eps}.
For applications to the problems
of rigidity in complex dynamics, see~\cite{dht},~\cite{maryaster},~\cite{Eps},~\cite{mak0},
\cite{Le},~\cite{mak},~\cite{mak1},~\cite{uz},~\cite{leij1},~\cite{le2}. 
In particular, in~\cite{Le} we use the operator $T$ for proving
the absence of invariant linefield on some Julia sets.
The scheme of~\cite{Le} is applied in~\cite{mak},~\cite{mak1},~\cite{uz}.
\end{com}
\begin{com}\label{natural}
The function $B_O$ associated to the periodic orbit $O$ appears
naturally in~\cite{Leij0} in the context of quadratic polynomials 
$f(z)=z^2+v$. Namely, assume that the
periodic orbit $O$ of period $n$ is attracting, more exactly, $0<|\rho|<1$. 
Consider a series
$H_\lambda(z)=\sum_{k> 0}\frac{\lambda^{k-1}}{(f^{k-1})'(v)(z-f^k(0))}$,
which converges for $|\lambda|<|\rho|^{1/n}$. 
Then the following identity holds 
(\cite{lman},~\cite{lpade}, see also~\cite{sy} ~\cite{LSY}):
$\lambda (TH_\lambda)(z)=H_\lambda(z)-\frac{D(\lambda)}{z-v}$,
where $D(\lambda)=\sum_{k> 0}\lambda^{k-1}/(f^{k-1})'(v)$. 
For a fixed $z$, the functions $H_\lambda(z)$
and $D(\lambda)$ extend to meromorphic functions on the
complex plane, with simple poles at the points $\lambda$, so that
$\lambda^n=\rho^{1-j}$, for $j=0,1,2,...$. Then calculations
show that the residue of $H_\lambda(z)$ at the point $\lambda=1$
is (up to a factor) the rational function $B_O(z)$.
Taking the residues at $\lambda=1$ of both sides of the above identity,
we come to the relation
$T B_O(z)=B_O(z)-\frac{L}{z-v}$, for some number $L$.
Furthermore, it is shown in~\cite{Leij0} that 
the latter holds for any periodic orbit of $f$ with $\rho\not=1$,
and that $L=\rho'(v)$. Surprisingly, all this is generalized by
the connection~(\ref{iruelle}) of the present paper
to a non-linear polynomial and rational function, with special
normalization at infinity.
The proof also sheds light on the nature of~(\ref{iruelle}).
\end{com}
The idea of the proof of~(\ref{iruelle}) is roughly as follows,
see Sects~\ref{red},~\ref{rthm} for details, for polynomials
and rational functions respectively. 
Given $f$
along with its periodic orbit $O$, we join it by a path inside of the
space associated to $f$ to a map $g$, such that $g$ is hyperbolic,
and the analytic continuation of $O$ along the path turns $O$ into
an attracting periodic orbit of $g$. Since both sides
of~(\ref{iruelle}) depend analytically on the local coordinates,
it follows that it is enough to prove~(\ref{iruelle}) for an open subset of
hyperbolic maps $g$ and for their attracting periodic orbits $O$.
To do this, we use the
theory of quasiconformal deformations (``Teichmuller theory'')
of rational maps developed in Mane-Sad-Sullivan~\cite{MSS} 
and McMullen-Sullivan~\cite{mcmsul}. 
The operator $T$ will serve as a
transfer between parameters and the dynamics in this context.

We show also that multipliers of non-repelling
periodic orbits are local coordinates in the
space associated to $f$ modulo a standard equivalence relation. 
For precise statements, see Theorem~\ref{attr} for polynomials,
and Sect.~\ref{rlocalsec} for rational functions.
We illustrate the results on the examples of 
quadratic polynomials (Comments~\ref{dhs}-\ref{cuspqu})
and quadratic rational functions (Corollary~\ref{corquad}).
The proof of Theorems~\ref{attr} and~\ref{rattr}
is based on~(\ref{iruelle}), see Sects.~\ref{attrs} and~\ref{rattrs}
respectively.
Then it boils down to the fact that $T$ has no fixed points 
of a certain form (which is roughly a linear combination
of functions $B_O$ for non-repelling orbits),
and this follows from the contraction property of $T$. 
The latter idea goes back to Thurston's work mentioned above
and has been applied, among others, in
\cite{dht}, \cite{maryaster}, \cite{Ts}, \cite{Eps}, \cite{Le}, \cite{Leij0}, 
\cite{mak}. 
See end of Sect.~\ref{mainform}, Sect.~\ref{attrs} and particularly
Sect.~\ref{rattrs} for precise formulations and further discussion.

In Sects.~\ref{s1}-\ref{attrs} we accomplish these aims for polynomials, see 
Theorems~\ref{main},~\ref{attr}, and in 
Sects.~\ref{rs1}-\ref{rattrs} we do this for  
rational maps, see Theorem~\ref{rmain} and Sect.~\ref{rlocalsec}.
Although the proof for polynomials and rational maps is essentially the same,
we treat the polynomial case separately
because of a special characteristic behavior of polynomials at $\infty$,
and also because the proof
in this case is technically simpler and hence more transparent. 
\begin{com}\label{field}
Since~(\ref{iruelle}) is a formal identity, 
which holds
for any rational function over ${\bf C}$ and any periodic orbit
with multiplier not equal to one, it holds (literally) 
for rational functions over every field which is isomorphic to ${\bf C}$,
in particular, for the $p$-adic fields.
\end{com} 
{\bf Acknowledgments.}  
Most of the paper was written during my stay
at the Institute of Mathematics of PAN, Warsaw, April-July, 2008
(arXiv 0809.0379, 2008).
I gratefully acknowledge this Institution for hospitality and 
excellent conditions for work. 
I am indebted to Alex Eremenko and Nessim Sibony for very useful
discussions on the proofs of Proposition~\ref{rlocal}.
I am grateful to  
Xavier Buff, David Kazhdan, Jan Kiwi, Tan Lei, Juan Rivera-Letelier,
Ehud de Shalit
for various discussions and questions, and
Maxim Kontsevich for the reference~\cite{conn}.
Some results of the present work should be related
to~\cite{Epstalk}. I thank Adam Epstein for
sending me~\cite{Epstalk} as well as some useful comments. 
Finally, I would like to thank
the referee for a constructive and detailed report.

\

Throughout the paper,
$$B(a,r)=\{z: |z-a|<r\}, \ \ B^*(R)=\{z:
|z|>R\}.$$
\section{Polynomials. Formulation of main results}\label{s1}
\subsection{Polynomial spaces}\label{space}
Introduce a space $\Pi_{d,\bar p}$ of polynomials and its subspace
$\Pi^q_{d, \bar p}$ as follows. Roughly speaking, the first space
is the set of maps with fixed multiplicities of critical points,
and maps from its subspace are those with a fixed number of
different critical values.
\begin{defi}\label{pol}
Let $d\ge 2$ be an integer, $\bar p$ a set of $p$ positive
integers $\bar p=\{m_j\}_{j=1}^p$, such that $\sum_{j=1}^p
m_j=d-1$, and $q$ an integer $1\le q\le p$. A polynomial $f$ of
degree $d$ belongs to $\Pi_{d,\bar p}$ iff it is monic and
centered, i.e., has the form
$$f(z)=z^d+a_1 z^{d-2}+...+a_{d-1},$$
and, moreover, $f$ has $p$ geometrically different critical points
$c_1,...,c_p$ with the multiplicities (as roots of $f'$)
$m_1,...,m_p$ resp.

The space $\Pi_{d, \bar p}^q$ is said to be a subset of those
$f\in \Pi_{d,\bar p}$, such that $f$ has precisely $q$
geometrically different critical values, i.e., the set
$\{v_j=f(c_j), \ j=1,...,p\}$ contains $q$ different points.
\end{defi}
In particular, $m_j=1$ iff $c_j$ is simple, so that $p=d-1$ iff
all critical points are simple. At the other extreme case, the
space $\Pi_{d, \bar 1}$ consists of the unicritical family
$z^d+v$.
Up to a linear change of variables, every non-linear polynomial
belongs to some $\Pi^q_{d, \bar p}$.

By definition,
\begin{equation}\label{eqcr}
f'(z)=d . \Pi_{j=1}^p(z-c_j)^{m_j}.
\end{equation}
Since $f$ is centered, there is a linear connection between
$c_1,...,c_p$:
\begin{equation}\label{lincr}
\sum_{j=1}^p m_j c_j=0.
\end{equation}
Let us identify $f\in \Pi_{d,\bar p}$ as above with the point
$$\bar f=\{a_1,...,a_{d-1}\}\in {\Ci}^{d-1}.$$
Given $f_0\in \Pi_{d, \bar p}$, we introduce two local coordinates
$\bar c$ and $\bar v$ in a neighborhood of $f_0$ in $\Pi_{d,\bar
p}$. Roughly speaking, $\bar c$ encodes a map through
geometrically different critical points, and $\bar v$ through
their images (corresponding critical values). Let us define it
precisely.

Let $\{c_1(f_0),...,c_p(f_0)\}$ be the collection of all
geometrically different critical points of $f_0$. We fix the order
of the critical points (so called "marked polynomial"). Then, for
every $f\in \Pi_{d, \bar p}$, such that $\bar f$ and $\bar f_0$
are close enough points of ${\Ci}^p$, the critical points
$c_1(f),...,c_p(f)$ of $f$ can be ordered in such a way, that
$c_j(f)$ is close to $c_j(f_0)$, $j=1,...,p$. We set
\begin{equation}\label{c}
\bar c(f)=\{f(0),c_1,...,c_p\},
\end{equation}
where $c_j=c_j(f)$. Note that $c_1,...,c_p$ satisfy~(\ref{lincr}).
It follows from the equality
$f(z)=f(0)+d \int_{0}^z \Pi_{j=1}^p (w-c_j)^{m_j}dw$,
that $f$ is defined uniquely by $\bar c$ and, moreover, $\bar f\in
{\Ci}^{d-1}$ is a holomorphic function of $\bar c$ modulo the
linear relation~(\ref{lincr}). Clearly, $\bar c$ lies in the
$p$-dimensional linear subspace of ${\Ci}^{p+1}$, and so we
identify this subspace with ${\Ci}^p$.

We are going to introduce a second coordinate system $\bar v$ in
a neighborhood of $f_0$, and prove that it is indeed a local
coordinate. In fact, $\bar v$ will play a central role for us. In
the previous notations, we define
\begin{equation}\label{v}
\bar v=\bar v(f)=\{v_1,...,v_p\}, \ \ \ v_j=\bar v_j(f)=f(c_j(f)).
\end{equation}
Let us stress that some critical values $v_j$ might coincide, so
that the number $q(f)$ of geometrically different critical values
of $f$ might be less than $p$. Nevertheless, the point $\bar v$ is
a point of ${\Ci}^p$.

We have a well-defined correspondence $\pi: \bar c\mapsto \bar v$ 
from a neighborhood of the point
$\bar c(f_0)$ in ${\Ci}^p$ into a neighborhood of $v(f_0)=\pi
(\bar c(f_0))$ in ${\Ci}^p$.
\begin{prop}\label{local}

1. The map
$$\pi: \bar c \mapsto \bar v$$
is a local biholomorphic homeomorphism of a neighborhood of $\bar
c(f_0)$ in ${\Ci}^p$ onto a neighborhood of $\bar v(f_0)$.

2. $\bar f$ is a local holomorphic function of $\bar v$. It is
locally biholomorphic if all the critical points of $f_0$ are
simple.
\end{prop}
We prove it in Sect.~\ref{pr1}. 

The space $\Pi_{d, \bar p}$ can therefore be identified in a
neighborhood of its point $f_0$ with a neighborhood
$W_{f_0}(\epsilon)= \{(v_1,...,v_p): |v_j-v_j(f_0)|<\epsilon\}$,
$\epsilon>0$, in ${\Ci}^p$. If, moreover, $f_0\in \Pi_{d, \bar
p}^q$, for some $q$, then its neighborhood in $\Pi_{d, \bar p}^q$
is the intersection of $W_{f_0}(\epsilon)$ with a $q$-dimensional
linear subspace of ${\Ci}^p$ defined by the conditions: $v_i=v_j$
iff $v_i(f_0)=v_j(f_0)$. Thus the vector $\{V_1,...,V_q\}$ of
different critical values of $f\in \Pi_{d, \bar p}^q$ serves as a
local coordinate systems in $\Pi_{d, \bar p}^q$.
\subsection{Main results}\label{mainform}
Let $f$ be a polynomial.
Consider a periodic orbit $O=\{b_k\}_{k=1}^n$ of $f$ of
exact period $n$. Denote its multiplier by $\rho$:
$$\rho=(f^n)'(b_k)=\Pi_{j=1}^n f'(b_j).$$
We assume that $\rho\not=1$. (For $\rho=1$, see next Subsect.~\ref{cusp}.)
Suppose $f$ is monic and centered. Then it belongs to some 
space $\Pi_{d, \bar p}$ and, moreover, to its subspace $\Pi^q_{d,
\bar p}$. (In fact, these spaces are defined uniquely, up to the
order of the different critical points.) These will be the
parameter spaces associated to $f$. By the Implicit Function
theorem and by Proposition~\ref{local}, there is a set of $n$
functions $O(\overline v)=\{b_k(\overline v)\}_{k=1}^n$ defined
and holomorphic in $\overline v\in \Ci^{p}$ in a neighborhood of
$\bar v(f)$, such that $O(\overline v)=O$ for $\overline
v=\overline v(f)$, and $O(\bar v)$ is a periodic orbit of $g\in
\Pi_{d, \bar p}$ of period $n$, where $g$ is in a neighborhood of
$f$, and $\overline v=\overline v(g)$. In particular, if
$\rho(\overline v)$ denotes the multiplier of the periodic orbit
$O(\overline g)$ of $g$, it is a holomorphic function of
$\overline v$ in this neighborhood. The standard notation
$\partial \rho/\partial v_j$ denotes the partial derivatives.

Now, suppose $g$ stays in a neighborhood of $f$ inside of $\Pi_{d,
\bar p}^q$. Then the multiplier $\rho$ of $O(g)$ is, in fact, a
holomorphic function of the vector of $q$ different critical
values $\{V_1,...,V_q\}$ of $g$. By $\partial^V\rho/\partial V_k$
we then denote the partial derivatives of $\rho$ w.r.t. these
critical values. We have:
\begin{equation}\label{vV}
\frac{\partial^V\rho}{\partial V_k}=\sum_{j: v_j=V_k}
\frac{\partial \rho}{\partial v_j}.
\end{equation}

The operator $T$ associated to $f$ and the
rational function $B=B_O$ associated to $O$ are defined in the
Introduction. 
\begin{theo}\label{main}
Suppose $f\in \Pi_{d, \bar p}$, so that $c_1,...,c_p$ denote all
geometrically different critical points of $f$, and $v_j=f(c_j)$
corresponding critical values (not necessarily different). Let $O$
be a periodic orbit of $f$ with multiplier $\rho\not=1$, and
$B=B_O$. Then 
\begin{equation}\label{ruelle}
B(z)-(TB)(z)= \sum_{j=1}^{p}\frac{\partial \rho}{\partial
v_j}\frac{1}{z-v_j},
\end{equation}
and
\begin{equation}\label{partialv}
\frac{\partial \rho}{\partial v_j}=-\frac{1}{(m_j-1)!}
\frac{d^{m_j-1}}{dw^{m_j-1}}|_{w=c_j}\frac{B(w)}{Q_j(w)}=
-\frac{1}{2\pi i}\int_{|w-c_j|=r}\frac{B(w)}{f'(w)}dw,
\end{equation}
where $c_j$ is a critical point of $f$ of multiplicity $m_j$,
and $Q_j$ is a polynomial defined by $f'(z)=(z-c_j)^{m_j}Q_j(z)$.
Furthermore, if $f\in \Pi_{d, \bar p}^q$, then
\begin{equation}\label{ruelleV}
B(z)-(TB)(z)= \sum_{k=1}^{q}\frac{\partial^V \rho}{\partial
V_k}\frac{1}{z-V_k},
\end{equation}
where $V_k$, $k=1,...,q$, are all pairwise different critical
values of $f$.
\end{theo}
In view of~(\ref{vV}), the formula~(\ref{ruelleV}) is an immediate
corollary of~(\ref{ruelle}).

For the family of unicritical polynomials, i.e., in the space
$\Pi_{d, \bar 1}$, the formula~(\ref{ruelle}) appears for the
first time in~\cite{Leij0}. Note that its proof in~\cite{Leij0} is
formal, and in this sense "mysterious", as it is pointed out
there. Our proof resolves in some way this "mystery". It is based
on the Teichmuller theory of rational
maps~\cite{mcmsul},~\cite{MSS}. The bridge between this theory and
our problem is provided by the following property of the operator
$T$ from the space $L_1({\Ci})$ into itself: the adjoint operator
of $T$ is an operator $T^*$ in $L_\infty({\Ci})$
acting as follows: $T^*\nu=|f'|^2/f'^2 \nu\circ f$.
A fixed point $\nu$ of $T^*$ is called an {\it invariant Beltrami form} 
of $f$, see~\cite{MSS},~\cite{mcmsul}.
To be more precise, we will make use of the following. 
A backward invariant Beltrami form on a set $V$ is a function
$\mu\in L_\infty(V)$, for which
$\mu(f(x))|f'(x)|^2/(f'(x))^2=\mu(x)$, for a.e. $x\in f^{-1}(V)$ .
For every function $\psi$, which is integrable on $V$, we then
have (by change of variable):
\begin{equation}\label{invruelle}
\int_{V}\mu(z) T\psi(z)d\sigma_z=\int_{f^{-1}(V)}\mu(z)
\psi(z)d\sigma_z.
\end{equation}
Here and below $d\sigma_z$ denotes the area element on the $z$-plane.
We have similarly that 
$T$ is a contraction in a sense that
\begin{equation}\label{rcontra}
\int_{V}|T\psi|d\sigma\le\int_{f^{-1}(V)}|\psi|d\sigma.
\end{equation}
\subsection{Cusps}\label{cusp}
Here we consider the remaining case $\rho=1$, under the
assumption that the periodic orbit is
{\it non-degenerate}.
In other words, we assume that the 
periodic orbit $O=\{b_1,...,b_n\}$ of $f$ of the exact
period $n$ has the multiplier $\rho=1$, and $(f^n)''(b_j)\not=0$,
for some (hence, for any) $j=1,...,n$.
Then, for any polynomial $g$, which is close to $f$, the map
$g$ in a small neighborhood of $O$ has either precisely two different periodic
orbits $O^\pm_g$ of period $n$ with multipliers
$\rho^\pm\not=1$, or precisely one periodic orbit $O_g$
of period $n$ with the multiplier $1$.

Now, assume that $f\in \Pi_{d, \bar p}^q$, and let $f_i$, $i=1,2,...$,
be any sequence of maps from $\Pi_{d, \bar p}$, such that
$f_i\to f$, $i\to \infty$. We assume that each $f_i$ has a periodic orbit
$O_i$ near $O$, such that
its multiplier $\rho_i\not=1$. In other words, 
the orbit $O_i$
is either $O^+_{f_i}$ or $O^-_{f_i}$.
Introduce
\begin{equation}\label{b-cusp}
\hat B_i(z)=(1-\rho_i)B_{O_i}(z)=
\sum_{b\in O_i} \{\frac{\rho_i(1-\rho_i)}{(z-b)^2}+
\frac{(f_i^n)''(b)}{z-b}\}.
\end{equation}
As $i\to \infty$, we have obviously that $\hat B_i$ tend
to a rational function $\hat B=\hat B_O$, where
$$\hat B(z)=\sum_{b\in O}\frac{(f^n)''(b)}{z-b}.$$
Now, multiplying both hand-sides of~(\ref{ruelle})
for $f_i$ and $B_{O_i}$ by $1-\rho_i$ and passing to the
limit as $i\to \infty$, we get: 
\begin{prop}\label{maincusp}
For every $j=1,...,p$, the following finite limit exists:
\begin{equation}\label{partialneutral}
\frac{\hat{\partial} \rho}{\partial v_j}
:=\lim_{i\to \infty}(1-\rho_i)\frac{\partial \rho_i}{\partial v_j}.
\end{equation} 
Then we have:
\begin{equation}\label{ruelleneutral}
\hat B(z)-(T\hat B)(z)= \sum_{j=1}^{p}\frac{\hat{\partial} \rho}{\partial
v_j}\frac{1}{z-v_j}.
\end{equation}
The formula~(\ref{partialv}) holds, where one
replaces $\rho$ and $B$ by $\hat\rho$ and $\hat B$
respectively.
Furthermore, if $f$ and $f_i$ are in $\Pi_{d, \bar p}^q$, then,
for every $j=1,...,p$, there exists a finite limit
\begin{equation}\label{partialneutralV}
\frac{\hat{\partial}^V \rho}{\partial v_j}
:=\lim_{i\to \infty}(1-\rho_i)\frac{\partial^V \rho_i}{\partial v_j},
\end{equation} 
and
\begin{equation}\label{ruelleVcusp}
\hat B(z)-(T\hat B)(z)= \sum_{k=1}^{q}\frac{\hat{\partial}^V \rho}{\partial
V_k}\frac{1}{z-V_k}.
\end{equation}
\end{prop}
\begin{com}\label{notzero}
We will see (take $r=1$ in the next Theorem~\ref{attr}) that
the vectors $\{\hat{\partial} \rho/\partial v_j\}_{j=1}^p$
and $\{\hat{\partial}^V \rho/\partial V_j\}_{j=1}^q$ are non-zero.
\end{com}
\subsection{Multipliers and local coordinates}
Theorem~\ref{main}, Proposition~\ref{maincusp} 
and the contraction property of $T$ yield the following.
\begin{theo}\label{attr}
Suppose that $f\in \Pi_{d, \bar p}^q$ has a collection
$O_1,...,O_r$ of $r$ different periodic orbits with the
corresponding multipliers $\rho_1,...,\rho_r$, such that each
$O_j$ is non-repelling: $|\rho_j|\le 1$, $j=1,...,r$. 
Assume that, if, for some $j$, $\rho_j=1$, then the periodic
orbit $O_j$ is non-degenerate. Denote by
$\tilde{\partial}^V \rho_j/\partial V_k$ the 
$\partial^V \rho_j/\partial V_k$ iff $\rho_j\not=1$ and
$\hat\partial^V \rho_j/\partial V_k$ iff $\rho_j=1$.
With these notations,  
introduce the following matrix ${\bf O}$:
\begin{equation}\label{mdelta}
{\bf O}=(\frac{\tilde{\partial}^V \rho_j}{\partial V_1},...,
\frac{\tilde{\partial}^V \rho_j}{\partial V_q})_{1\le j\le r}.
\end{equation}
Assume furthermore that, for every $j=1,...,r$, either $\rho_j\not=0$, or, if
$\rho=0$, then the periodic orbit $O_j$ contains a single critical
point, and this critical point is simple. Then the rank of the
matrix ${\bf O}$ is equal to $r$, that is, maximal.
\end{theo}
\begin{com}\label{fatou}
By the Fatou-Douady-Shishikura inequality, see~\cite{CG},~\cite{shi}, 
$r\le q$, the number of
geometrically different critical values of the polynomial $f$.
See also Comment~\ref{last} as well as Comments~\ref{rfatou} and~\ref{rlast}.
\end{com}
\begin{com}\label{newcoord}
Assume in Theorem~\ref{attr} that $\rho_j\not=1$
for $j=1,...,r$.
Applying in this case 
the Implicit Function theorem to the matrix ${\bf O}$ of
rank $r$, we obtain that one can define a new local coordinate system in
$\Pi_{d, \bar p}^q$ by replacing $r$ coordinates in $\bar
V=(V_1,...V_q)$ by $r$ multipliers of different non-repelling
periodic orbits with the multipliers not equal to $1$.
In fact, the case when some $\rho_j=1$ can also be included
replacing $\rho_j$ by $-\{(1-\rho_j^+)^2+(1-\rho_j^-)^2\}/4$
in a neighborhood of $f$, where $\rho_j^{\pm}$ are the multipliers
of the periodic orbits $O_j^{\pm}$ of nearby maps
defined in the previous Sect.~\ref{cusp}.
\end{com}
\begin{com}\label{dhs}
Theorem~\ref{attr} contains as a particular case the 
Douady-Hubbard-Sullivan Theorem: the multiplier map of an
attracting periodic orbit of the map $z^2+v$ is an isomorphism of
the corresponding hyperbolic component of the Mandelbrot set onto the unit
disk~\cite{DH1},~\cite{CG},~\cite{mcmsul}. For other generalizations of this
important result for polynomials, see~\cite{milpol}. Note that the case $f\in
\Pi_{d, \bar p}^q$ and $r=q$, under the assumption that the
periodic orbits $O_1,...,O_r$ are attracting follows also from a
general result on hyperbolic polynomials proved in~\cite{milpol}.
Note however that the method of quasiconformal surgery used
in~\cite{DH1},~\cite{CG},~\cite{mcmsul},~\cite{milpol} breaks down in 
the presence
of a neutral periodic orbit. Our result is completely general. On
the other hand, it is local.
\end{com}
\begin{com}\label{cuspqu}
Consider another particular case:
$f_{v_0}(z)=z^d+v_0$ and $O_0$ is a non-degenerate periodic orbit
of $f_{v_0}$ with the multiplier $\rho_0=1$. Then the matrix ${\bf O}$
is one-dimensional and consists of
the number 
\begin{equation}\label{neutralqu}
\hat\rho'(v_0)
:=\lim_{v\to v_0}(1-\rho(v))\rho'(v),
\end{equation} 
where $\rho(v)$ is the multiplier of a periodic orbit
of $f_v(z)=z^d+v$, $v\not=v_0$, which is close to $O_0$.
By Theorem~\ref{attr}, $\hat\rho'(v_0)\not=0$, which means
that the corresponding hyperbolic component
of the connectedness locus of the family $f_v$ has a
cusp at the point $v_0$. The latter is proven in~\cite{DH2} (for $d=2$)
using global considerations. 
\end{com}
\section{Theorem~\ref{main}}\label{thm}
\subsection{Proof of Proposition~\ref{local} }\label{pr1}
The map $\pi$ is locally well-defined and holomorphic,
because the coefficients of $f$ are holomorphic functions of $\bar
c$. It maps a neighborhood of $\bar c$ in ${\bf C}^p$ into
${\bf C}^p$. Therefore, to prove that $\pi$ is locally biholomorphic,
it is enough to show that $\pi$ is a local
injection (see e.g.~\cite{whit}, Chapter 4, Theorem 1V).
On the other hand, the latter follows essentially from
the Monodromy Theorem. Here is a detailed proof.

Fix $f_0\in\Pi_{d, \bar p}$. It has $p$ geometrically different
critical points $c_j(f_0)$ and $q_0=q(f_0)$ geometrically
different critical values $v_1^0,...,v^0_{q_0}$, $q_0\le p$.
Choose a covering of the Riemann sphere $\hat{\Ci}$ by a finite
collection of (open) balls $B_1,...,B_m$ centered at some points
$a_1,...,a_m$, as follows.

(1) For $1\le k\le q_0$, the ball $B_k$ is centered at the
critical value $v^0_k$ of $f_0$, i.e., $a_k=v^0_k$, and
the closures $\bar{B_k}, \bar{B_i}$, for $1\le i<k\le q_0$, are pairwise
disjoint.

(2) $B_m$ is centered at infinity: $a_m=\infty$. For every $1\le
k\le m$, the center $a_k$ of $B_k$ is away from the closure $\bar
B_i$ of any other ball $B_i$, $k\not=i$.

(3) For every $1\le k\le m$, and every component $U$ of
$f_0^{-1}(B_k)$, the following holds: either $U$ is disjoint from
the set of critical points of $f_0$, or $U$ contains one and only
one critical point $c_j(f_0)$ of $f_0$. In the former case, the
map $f_0: U\to B_k$ is univalent, and in the latter case, $f_0:
U\to B_k$ is an $m_j+1$-branched covering, with a single critical
point at $c_j(f_0)$.

By this, every component $U$ of $f_0^{-1}(B_k)$ contains one and
only one preimage $w_U$ of $a_k$ by $f_0$. Call $w_U$ the "center"
of the component $U$.
Denote by $d(z,w)$ the spherical distance between $z,w$ in the
Riemann sphere.
Consider any $f\in \Pi_{d,\bar p}$, such that $\bar f$ and $\bar
f_0$ are close points in ${\Ci}^{d-1}$. By definition, $f$ has
$p$ geometrically different critical points $c_j(f)$ with the
corresponding multiplicities $m_j$, and $c_j(f)$ is close to
$c_j(f_0)$, $1\le j\le p$. Note however, that the number $q(f)$ of
geometrically different critical values of $f$ can be larger than
the number $q_0$ of geometrically different critical values of
$f_0$. We have, by the above, similar properties for preimages of
$B_k$ by $f$:

(1f) For every $1\le k\le m$, and every component $U(f)$ of
$f^{-1}(B_k)$ the following holds. The set $U(f)$ contains one and
only one "center" $w_U$ of some component $U$ of $f_0^{-1}(B_k)$.
Moreover, $U(f)$ and $U$ are close (in, say, Hausdorff
distance).

(2f) There are two possibilities: (a) if $w_U\in U(f)$ is not a
critical point of $f_0$, then the map $f: U(f)\to B_k$ is
univalent, (b) if $w_U=c_j(f_0)$, for some $j$, then $f: U\to B_k$
is an $m_j+1$-branched covering, with the single critical point at
$c_j(f)$.

To prove the injectivity, consider two maps $f_1$, $f_2$ in
$\Pi_{d, \bar p}$, so that $\bar c(f_1)$, $\bar c(f_2)$ are close
to $\bar c(f_0)$, and assume that
\begin{equation}\label{v1=v2}
\bar v(f_1)=\bar v(f_2).
\end{equation}
We know that $\bar f$ is a continuous (even holomorphic)
function of $\bar c$. Hence, $\bar f_1$, $\bar f_2$ are close to
$\bar f_0$, too. We get from (1f)-(2f):

(3f) fix $r>0$ small enough (smaller than half of the spherical
distance between any point $w_U$ and any component of
$f_0^{-1}(B_k)$ other than $U$). For every $1\le k\le m$, there
exists $\epsilon_k>0$ with the following property. For a component
$U$ of $f_0^{-1}(B_k)$, if $\bar f_1$, $\bar f_2$ are close enough
to $\bar f_0$, and if there is a point $\hat z\in U$, such that
$d(\hat z, w_U)>r$ and $d(F\circ f_1(\hat z),\hat
z)<\epsilon_k/2$, for some branch $F$ of $f_2^{-1}$ in $B_k$,
then:

(i) $F\circ f_1$ is a (well-defined) holomorphic function in the
component $U(f_1)$ which contains $w_U$,

(ii) $d(F\circ f_1(z),z)<\epsilon_k$, for all $z\in U(f_1)$.

Now, let $\bar f_i$, $i=1,2$ be close enough to $\bar f_0$. Fix
$\epsilon_*$ positive and less than $\epsilon_k/2^{md}$, for all
$k$. Let us start with a branch $F_\infty(z)=z^{1/d}+...$ of $f_2^{-1}$
in $B_m$, such that $g=F_\infty\circ f_1$ is well-defined near infinity,
and $g(z)=z+O(1/z)$ at $\infty$.
Then $g$ extends to a holomorphic function in $U_\infty(f_1)$. 
Since $\bar f_i$, $i=1,2$, and $\bar
f_0$ are close enough, then $d(g(\hat z),\hat z)<\epsilon_*$, for
any $\hat z$ in the intersection of $U_\infty(f_1)$ and any
component $V(f_1)$ of any other $f_1^{-1}(B_k)$. By (3f), $g$
extends to a holomorphic function along every chain of components of
$f_1^{-1}(B_k)$ that form a connected set, therefore, $g$ is
holomorphic in $\Ci$. By the normalization, it is the identity map, which
proves that $f_1=f_2$.

The second part is obvious because if the (finite) 
critical points of $f$ are simple, then $\bar f$ is a local
biholomorphic map of $\bar c$. By the first part, we are done.
\subsection{Reductions}\label{red}
By the continuity of functions $\rho$ and $\partial \rho/\partial
v_j$ in $\bar v$, it is enough to prove the formulae of
Theorem~\ref{main} assuming that $\rho\not=0$.
\paragraph{The identity.}
Denote
\begin{equation}\label{A-func}
A(z)=\frac{1}{\rho} B(z)=\sum_{k=1}^n
\frac{1}{(z-b_k)^2}+\frac{1}{\rho(1-\rho)} \sum_{k=1}^n
\frac{(f^n)''(b_k)}{z-b_k}.
\end{equation}
First, we will prove the following general identity about 
the rational functions fixing infinity.
\begin{theo}\label{rformula}
Let $f$ be any rational function so that $\infty$ is a fixed
point (possibly, superattracting) of $f$. Let $c_j$, $j=1,...,p$ be all
geometrically different finite critical points of $f$, and assume
that the corresponding critical values $v_j=f(c_j)$, $j=1,...,p$,
are also finite. Then there are numbers $L_1,...,L_p$, such that
\begin{equation}\label{rruellehyp}
A(z)-(TA)(z)= \sum_{j=1}^{p}\frac{L_j}{z-v_j}.
\end{equation}
For $j=1,...,p$, the coefficient
\begin{equation}\label{lj}
L_j=-\frac{1}{(m_j-1)!}\frac{d^{m_j-1}}{dw^{m_j-1}}|_{w=c_j}(\frac{A(w)}{Q_j(w)}),
\end{equation}
where $Q_j$ is a local analytic function near $c_j$ defined by
$f'(z)=(z-c_j)^{m_j}Q_j(z)$, so that $Q_j(c_j)\not=0$. In
particular, if $c_j$ is simple, then
$L_j=-A(c_j)/f''(c_j)$.
\end{theo}
\paragraph{Reduction to the hyperbolic case.}
Here we show that it is enough to prove Theorem~\ref{main} only
for those $f$ from $\Pi_{d, \bar p}$ that satisfy the following
conditions:

(1) $f$ is a hyperbolic map, moreover, $O$ is an attracting
periodic orbit of $f$, which attracts all critical points $c_j$,
$j=1,...,p$,

(2) $f$ has no critical relations between critical points except
for the constant multiplicities of the critical points themselves:
$f^n(c_i)=f^m(c_j)$ if and only if $i=j$ and $m=n$.

Indeed, assume that Theorem~\ref{main} holds for this subset of
maps from $\Pi_{d, \bar p}$. Note that it is open in $\Pi_{d, \bar
p}$. Given now any $f$ as in Theorem~\ref{main}, we find a real
analytic simple path $\gamma: [0,1]\to \Pi_{d, \bar p}$,
$\gamma(t)=g_t$, which obeys the following properties:
(i) $g_0=f$,
(ii) $g_1$ satisfies conditions (1)-(2),
(iii) the analytic continuation $O_t$ (a periodic orbit of $g_t$)
of the periodic orbit $O$ along the path is well-defined (i.e., the
multiplier of $O_t$ is not $1$ for $t\in [0,1]$), and $O_1$, a
periodic orbit of $g_1$, is attracting. All critical points of
$g_1$ are attracted to $O_1$.

Let us for a moment take for granted the existence of such path.
Since the critical points of $g_t=\gamma(t)$ change continuously
along $\gamma$ and don't collide, the coordinate system $\bar v$
can be chosen changing continuously in a whole neighborhood of
$\gamma$. Fix $z$. Denote by $\Delta(z, \overline v)$ the
difference between the left and the right hand sides
of~(\ref{ruelle}). It is an analytic function in $\bar v$ in a
neighborhood of every point $\bar v(g_t)$, $t\in [0,1]$. On the
other hand, by the assumption that the theorem holds for maps
satisfying (1)-(2), it is identically zero in a neighborhood of
$\bar v(g_1)$. By the Uniqueness Theorem for analytic functions,
$\Delta(z, \bar v)=0$.

Let us show that the path $\gamma$ as above does exist. First, we
need the following fact about the parameter space of the
unicritical family $p_c(z)=z^d+c$. Consider any repelling periodic
orbit $Q$ of the map $p_0(z)=z^d$. There is a real analytic simple path
$p_{c_Q(t)}$, $c_Q:[0,1]\to \Ci$, such that there exists an
analytic continuation of $Q$ to a periodic orbit $Q_t$ of
$p_{c_Q(t)}$, $0\le t\le 1$, such that $Q_0=Q$ and $Q_1$ is
attractive. Indeed, if we assume the contrary that such a path
does not exist, then the Monodromy Theorem ensures the
existence of an analytic continuation of $Q$ to the whole complex
plane. Then the multiplier of this continuation is an entire
function in $c$ that omits values in the unit disk, which is
impossible. Thus the path $p_{c_Q(t)}$ as above does exist.

Note also that it is easy to find maps from $\Pi_{d, \bar p}$ near
every $p_c$.

Let us come back to the construction of the path $\gamma$.
Firstly, we connect $f$ to the map $p_0$ by a path $\gamma_0$ from
$[0,1]$ to the space of polynomials of degree $d$, so that
$\gamma_0(t)\in \Pi_{d, \bar p}$ for $0\le t<1$. We construct
$\gamma_0$ explicitly as follows. For $\tau\in \Ci$, set
$c_j(\tau)=(1-\tau)c_j(f)$, $1\le j\le p$. For every $\tau$,
define a polynomial
\begin{equation}\label{expl}
F_\tau(z)=(1-\tau)f(0)+d . \int_{0}^z
\Pi_{j=1}^p(w-c_j(\tau))^{m_j} dw.
\end{equation}
Note that $F_\tau\in \Pi_{d, \bar p}$ for every $\tau\not=0$. Then
a real analytic simple curve $\tau:[0,1]\to \Ci$, $\tau(0)=0$,
$\tau(1)=1$, can be chosen so that the analytic continuation $O_t$
of the periodic orbit $O$ of $f$ along the path
$\gamma_0(t)=F_{\tau(t)}$ exists, and $O_1$ is some periodic orbit
$Q$ of $p_0$. We proceed by a real analytic path $c_Q$ in the
parameter plane of $p_c$ that turns $Q$ into an attracting
periodic orbit of some $p_c$. Finally, we find the desired path
$\gamma$ in $\Pi_{d, \bar p}$ in a neighborhood of $c_Q\circ
\gamma_0$.
\paragraph{Hyperbolic maps}
Here we describe how to prove Theorem~\ref{main} for the
hyperbolic maps $f$. We assume that {\it $f\in \Pi_{d, \bar p}$
satisfies the conditions (1)-(2) of the previous paragraph.}
To clarify the meaning of the coefficients $L_j$ in~(\ref{rruellehyp}), 
we will use the
theory of quasiconformal deformations of rational maps~\cite{mcmsul}. The 
main technical
part is contained in the next Theorem~\ref{teich}. 
By a {\it Beltrami coefficient} we mean a measurable function $\nu(z)$
on the Riemann sphere, such that $|\nu(z)|\le k<1$ for almost every $z$.
Let $\nu(z, t)$ be an {\it analytic family of invariant Beltrami
coefficients}. By this we mean a family $\nu_t(z)=\nu(z, t)$
of Beltrami coefficients on the
Riemann sphere, which is
analytic in $t$ as a map from
a neighborhood of $t=0$ into $L_\infty(\C)$, 
and such that, for every $t$, $\nu_t$ 
is an invariant Beltrami form
of $f$. We always assume $\nu(z, 0)=0$.
Assume, additionally, $\nu(z,t)=0$ for $z$ in the basin of
infinity of $f$.
Let $h_t$ be an analytic in $t$ family of quasiconformal
homeomorphism in the plane tangent to $\infty$, so that $h_t$ has
the complex dilatation $\nu(z,t)$ 
(i.e., $\nu(t,z)=\frac{\partial h_t}{\partial \bar{z}}/
\frac{\partial h_t}{\partial z}$),
and $h_0=id$. Set
$f_t=h_t\circ f\circ h_t^{-1}$. It is well-known that then
$\{f_t\}$ is an analytic family of polynomials. Since $f_t$ is
conjugated to $f$, then $f_t\in \Pi_{d, \bar p}$. Let $O_t=h_t(O)$
be the corresponding attracting periodic orbit of $f_t$. Denote by
$\rho(t)$ its multiplier, and by $v_j(t)=h_t(v_j)$ the 
critical values of $f_t$.
Based on Theorem~\ref{rformula}, we derive:
\begin{theo}\label{teich}
\begin{equation}\label{teichpol}
\frac{\rho'(0)}{\rho}=\sum_{j=1}^{p} L_j v_j'(0).
\end{equation}
\end{theo}
\paragraph{Concluding argument.} Here we show how Theorem~\ref{teich}
implies Theorem~\ref{main}.
We are going to compare~(\ref{teichpol}) to the following obvious
identity:
\begin{equation}\label{teichformalpol}
\rho'(0)=\sum_{j=1}^{p} \frac{\partial\rho}{\partial v_j}v_j'(0).
\end{equation}
The proof will be finished once we will show that the vector
$\{v_1'(0),...,v_{p}'(0)\}$ for $f\in \Pi_{d, \bar p}$ can be
chosen arbitrary.

For every vector $\bar v'=\{v_1',...,v_{p}'\}\in {\Ci}^{p}$ of
initial conditions, there exists an analytic family $f_t$ of
polynomials from $\Pi_{d, \bar p}$ with different critical points
$c_1(t),...,c_p(t)$ and corresponding critical values
$v_1(t),...,v_{p}(t)$, $v_j(t)=f_t(c_j(t)$, such that
$v_j'(0)=v_j'$, for $1\le j\le p$. Indeed, this is an immediate
consequence of Proposition~\ref{local}, where one can simply take
locally $\bar v(t)=\bar v+t \bar v'$, and find by the inverse
holomorphic correspondence $\bar v\mapsto \bar c$ a local family
$f_t$, such that $f_0=f$.
Recall that $f\in \Pi_{d, \bar p}$ is a hyperbolic polynomial,
which has no critical relations except for constant multiplicities
at the critical points and such that no critical point of $f$ is
attracted by $\infty$. Then so is the conjugated map $f_t$, for
every $t$ close to $0$. In particular, the basin of infinity of
$f_t$ is simply-connected on the Riemann sphere. We construct a
holomorphic motion $h_t$ of the plane as follows. First, for every
$t$, define $h_t$ in the basin of infinity of $f$ to be
$h_t=B_{f_t}^{-1}\circ B_f$, where $B_P$ denotes the Bottcher
coordinate function of a polynomial $P$ such that $B_P(z)/z\to 1$
as $z\to \infty$. Note that $h_t$ is holomorphic in the basin of
$\infty$. Then we define $h_t$ on the grand orbits of the critical
points of $f$ as in the proof of Theorem 7.4 of~\cite{mcmsul}. By
Theorem 3 of~\cite{BR}, the holomorphic motion $h_t$ extends in a
unique way to a holomorphic motion of the plane, which we again
denote by $h_t$, such that the complex dilatation of $h_t$ is
harmonic. As it is shown in~\cite{mcmsul}, this $h_t$ agrees with
the dynamics. By Theorem 3 of~\cite{BR}, the complex dilatation
$\nu(t,z)$
of $h_t$ depends holomorphically on $t$. It vanishes in the basin
of infinity of $f$ because $h_t$ is holomorphic there.

This proves the existence of $\nu(z,t)$ as above, which determines
$v_j(t)$ with prescribed values $v'_j(0)=v_j'$, $j=1,...,p$. By
this we finish the proof of the implication that
Theorem~\ref{teich} implies Theorem~\ref{main}.
\subsection{Proof of Theorem~\ref{rformula}}
\paragraph{Action of $T$ on Cauchy kernels.}
\begin{lem}\label{rT}
Let $f$ be as in Theorem~\ref{rformula}, and $a\in \Ci$ a
parameter. Assume all finite critical points $c_j$, $j=1,...,p$,
are simple, and the corresponding critical values $v_j=f(c_j)$ are
finite. Then
\begin{equation}\label{rcauchy1}
T\frac{1}{z-a}=\frac{1}{f'(a)}\frac{1}{z-f(a)}+\sum_{j=1}^{N}\frac{1}{f"(c_j)(c_j-a)}\frac{1}{z-v_j}.
\end{equation}
Moreover,
\begin{equation}\label{rcauchy2}
T\frac{1}{(z-a)^2}=\frac{1}{(z-f(a))^2}-
\frac{f"(a)}{(f'(a))^2}\frac{1}{z-f(a)}+\sum_{j=1}^{N}\frac{1}{f"(c_j)(c_j-a)^2}\frac{1}{z-v_j}.
\end{equation}
\end{lem}
\begin{proof}
Consider the integral
$$I=\frac{1}{2\pi i}\int_{|w|=R}\frac{dw}{f'(w)(f(w)-z)(w-a)}$$
and apply the Residue Theorem.
It gives~(\ref{rcauchy1}). Taking the derivative of
~(\ref{rcauchy1}) with respect to the parameter $a$, we get
~(\ref{rcauchy2}).
\end{proof}
\paragraph{Proof in the case of simple critical points.}
Recall that
$$A(z)=\sum_{k=1}^n \frac{1}{(z-b_k)^2}+\sum_{k=1}^n \frac{\gamma_k}{z-b_k},$$
where $O=\{b_1,...,b_n\}$ is a periodic orbit of $f$ of exact
period $n$ and with the multiplier $\rho\not=1,0$, and
$$\gamma_k=\frac{(f^n)"(b_k)}{\rho(1-\rho)}.$$
Assuming all critical points are simple, we can apply
Lemma~\ref{rT}, and see that
\begin{equation}\label{TA}
(TA)(z)=\sum_{k=1}^n\frac{1}{(z-f(b_k))^2}+\sum_{j=1}^n
\frac{\frac{\gamma_k}{f'(b_k)}-\frac{f"(b_k)}{(f'(b_k))^2}}{z-f(b_k)}
+\sum_{j=1}^{d-1}\frac{A(c_j)}{f"(c_j)}\frac{1}{z-v_j}.
\end{equation}
Therefore,
$$A(z)-(TA)(z)=\sum_{k=1}^n\frac{\gamma_{k+1}-\frac{\gamma_k}{f'(b_k)}+\frac{f"(b_k)}{(f'(b_k))^2}}
{z-b_{k+1}}-\sum_{j=1}^{d-1}\frac{A(c_j)}{f"(c_j)}\frac{1}{z-v_j},$$
where we assume that $\gamma_{n+k}=\gamma_k$, $b_{n+k}=b_k$. 
The proof of Theorem~\ref{rformula} in this case will be concluded
once we check the following:
$$\gamma_{k+1}-\frac{\gamma_k}{f'(b_k)}+\frac{f"(b_k)}{(f'(b_k))^2}=0.$$
One can assume $k=1$. We have:
$$\frac{\gamma_1}{f'(b_1)}-\frac{f"(b_1)}{(f'(b_1))^2}=
\frac{(f^n)"(b_1) f'(b_1)-f"(b_1)\rho(1-\rho)}{(f'(b_1))^2 \rho(1-\rho)}.$$
Now,
$$(f^n)"(b_1) f'(b_1)-f"(b_1)\rho(1-\rho)=
f'(b_1)(f"(b_1)\Pi_{j=2}^n f'(b_j)+$$
$$\sum_{k=2}^n f"(b_k)
((f^{k-1})'(b_1))^2\Pi_{j=k+1}^n f'(b_j))-f"(b_1)(\Pi_{j=1}^n
f'(b_j)-\Pi_{j=1}^n (f'(b_j))^2)=$$
$$f'(b_1)\sum_{k=2}^n f"(b_k)
((f^{k-1})'(b_1))^2\Pi_{j=k+1}^n f'(b_j)-f"(b_1)(\Pi_{j=1}^n
f'(b_j)-\Pi_{j=1}^n f'(b_j)-\Pi_{j=1}^n (f'(b_j))^2)=$$
$$f'(b_1)\sum_{k=2}^n f"(b_k) ((f^{k-1})'(b_1))^2\Pi_{j=k+1}^n
f'(b_j)+f"(b_1)\Pi_{j=1}^n (f'(b_j))^2=(f^n)"(b_2) (f'(b_1))^2.$$
\paragraph{Multiple critical points.}
Let a rational function $f$ be such that $f(z)=\sigma z^{m_\infty}+...$ at
$\infty$. Suppose $f$ has a critical point $c$ with multiplicity
$m>1$. This means that $f'(z)=(z-c)^m Q(z)$, where $Q$ is a rational
function, such that $Q(c)\not=0$. 
Let us approximate $f$ by a sequence of rational functions $f_n$,
such that $f_n(z)=\sigma_n z^{m_\infty}+...$ at $\infty$, and so that all
critical points of every $f_n$ are simple. In particular, there
are $m$ critical points $c_1(n),...,c_m(n)$ of $f_n$, such that
$c_j(n)\to c$ as $n\to \infty$, for each $j=1,...,m$. For every
$n$, $f_n$ has a periodic orbit $O_n$, so that $O_n\to O$ as $n\to
\infty$. Denote $A_n$ the function corresponding to $O_n$ and
$f_n$. Then
$$A_n(z)-T_{f_n}A_n(z)\to A(z)-TA(z)$$
as $n\to \infty$. Thus according to the proven case of simple
critical points the proof of the identity will be done if we show
that the following limit exists:
\begin{equation}\label{resmult}
\lim_{n\to \infty} \sum_{j=1}^m
\frac{A_n(c_j(n))}{f_n"(c_j(n))}\frac{1}{f_n(c_j(n))-z}=\frac{-L}{f(c)-z},
\end{equation}
for some $L$ and all $z$ with large modulus. We use again the
Residue theorem. Fix a small circle $C$ around $c$. Then, for
every big $n$ and $|z|$ large enough,
$$\lim_{n\to \infty}\frac{A_n(c_j(n))}{f_n"(c_j(n))}\frac{1}{f_n(c_j(n))-z}=
\lim_{n\to \infty}\frac{1}{2\pi i}\int_C
\frac{A_n(w)}{f_n'(w)(f_n(w)-z)}dw= \frac{1}{2\pi i}\int_C
\frac{A(w)}{f'(w)(f(w)-z)}dw.$$ An easy calculation shows that
$$\frac{A(w)}{f'(w)(f(w)-z)}=\frac{A(w)}{(w-c)^m
Q(w)((f(c)-z)+O((w-c)^{m+1}))}=$$
$$\frac{A(w)}{Q(w)(f(c)-z)}\frac{1}{(w-c)^m}+O(w-c)=
\sum_{k=0}^\infty \frac{B_k}{f(c)-z} (w-c)^{k-m}+O(w-c),$$ where
$B_k$ are defined by the Taylor expansion
$A(w)/Q(w)=\sum_{k=0}^\infty B_k (w-c)^k$. It gives
us~(\ref{resmult}) with
$$-L=B_{m-1}=\frac{1}{(m-1)!}\frac{d^{m-1}}{dw^{m-1}}|_{w=c}
(\frac{A(w)}{Q(w)}).$$
\subsection{Proof of Theorem~\ref{teich}}\label{th4}
\paragraph{Beltrami coefficients.}
As it has been mentioned already, we derive this theorem by making
use of quasiconformal deformations. The following fundamental 
facts about quasiconformal maps are well-known
(see e.g. \cite{Ah}) and will be used throughout the paper.
The Measurable Riemann Theorem states that, 
given a Beltrami coefficient $\nu$,
there exists a unique quasiconformal
homeomorphism $\psi^\nu$ of the plane with the complex dilatation $\nu$,
i.e., $\nu(z)=\frac{\partial \psi^\nu}{\partial \bar{z}}/
\frac{\partial \psi^\nu}{\partial z}$ a.e., 
and such that $\psi^\nu$ fixes $0, 1$, and $\infty$. Assume that
$\nu_t$ depends on a complex parameter $t$ and
\begin{equation}\label{invfam}
\nu_t(z)=t\mu(z)+t\epsilon(z,t), 
\end{equation}
where $\mu$, $\epsilon$ belong to $L_\infty$,
and $||\epsilon(z,t)||_\infty\to 0$ as $t\to 0$. Then there exists
\begin{equation}\label{ahl}
\frac{\partial \psi^{\nu_t}}{\partial
t}|_{t=0}(z)=-\frac{1}{\pi}\int_{\Ci}\mu(w)\frac{z(z-1)}{w(w-1)(w-z)}d\sigma_w.
\end{equation}
Let $f\in \Pi_{d, \bar p}$, and satisfy the conditions
(1)-(2) of Sect.~\ref{red}. Remember that $\nu(z, t)$ is an
analytic family of invariant Beltrami coefficients on $\C$,
such that $\nu(z,0)=0$ and $\nu(z,t)=0$ for $z$ in the basin of infinity of
$f$. Since $\nu(z,t)$ is differentiable at $t=0$,
$\nu(z,t)=t\mu(z)+t\epsilon(z,t)$,
where $||\epsilon(z,t)||_\infty\to 0$ as $t\to 0$.
Note that $\mu$ is invariant by $f$, too. Indeed, as $\nu(z,t)$
is $f$-invariant, for every $t$,
$$t\mu(f(z)) \frac{|f'(z)|^2}{f'(z)^2}+t\epsilon(t,f(z))
\frac{|f'(z)|^2}{f'(z)^2}=t\mu(z)+t\epsilon(z,t),$$ which,
together with $\epsilon(z,t)\to 0$ for $t\to 0$, implies that
$(|f'|/f')^2 \mu\circ f=\mu$.
It follows similarly that $\mu$ vanishes at the basin of
$\infty$, too (for $|t|$ small enough).

Let $h_t$ be an analytic family of quasiconformal
homeomorphism in the plane tangent to $\infty$, so that $h_t$ has
the complex dilatation $\nu(z,t)$, and $h_0=id$. Then
$f_t=h_t\circ f\circ h_t^{-1}$ is a family of
polynomials from $\Pi_{d, \bar p}$, which is
analytic in $t$. Denote by $O_t=h_t(O)$ the
corresponding attracting periodic orbit of $f_t$, by $\rho(t)$ its
multiplier, and by $v_j(t)=h_t(v_j)$ the critical values of
$f_t$.
\paragraph{Speed of the multiplier.}
We need a formula for the speed of the multiplier of a periodic
orbit in an analytic family of maps obtained by a quasiconformal
deformation. A similar formula is proved in~\cite{lbrin}. 
For completeness, we reproduce the proof here. It is based
on the formula~(\ref{ahl}). 
Let $b\in O$. 
A fundamental region $C$ near $b$ is a (measurable) set, such that
every orbit of the dynamics $z\mapsto f^n(z)$ near $b$ enters $C$
once. 
Note that then $f^{kn}(C)$, $k=1,2,...$, are again
fundamental regions (tending to $b$). 
Usually, $C$ will be a
domain bounded by a small simple curve that surrounds $b$ and its
image by $f^{n}$. 
\begin{lem}\label{linspeed} 
\begin{equation}\label{linspeedf}
\frac{\rho'(0)}{\rho}=-\frac{1}{\pi}\lim_{C\to \{b\}}\int_{C}
\frac{\mu(z)}{(z-b)^2}d\sigma_z,
\end{equation}
where $C$ is a fundamental region near $b$.
\end{lem}
\begin{proof}
Let us linearize $f^n$ near $b$
by fixing a disk $D=\{|z|<r_0\}$ and a univalent map  $K: D\to \C$,
such that $K(0)=b$ and $f^n\circ K=K\circ \rho$ in $D$.
Given a function $\tau$ defined near the point $b$ denote by 
$\hat \tau=|K'|^2/(K')^2\tau\circ K$ the pullback
of $\tau$ to $D$. For every $t$, 
the pullback $\hat \nu(w,t)$ of the Beltrami
coefficient $\nu(w,t)$ is invariant by the linear map 
${\bf \rho}: w\mapsto \rho w$, i.e., 
$\hat \nu(w,t)=|\rho|^2/\rho^2\hat \nu(\rho w,t)$. 
For every fixed $t$, extend $\hat \nu(w,t)$ to $\C$ by the latter equation.
Denote by
$\phi_t$ the quasiconformal map of the plane with the
complex dilatation $\hat\nu(w,t)$,
which fixes $0, 1$, and $\infty$. Then the map
$\phi_t\circ \rho\circ \phi_t^{-1}$ is again linear
$w\mapsto \lambda(t) w$, for some $|\lambda(t)|<1$.
It is easy to see from the construction that
$w\mapsto \lambda(t) w$ is analytically conjugate to $f_t^n$
near $h_t(b)$. Therefore, $\lambda(t)=\rho(t)$.
Note that
$\hat \nu(w,t)=t\hat \mu(w)+t\hat \epsilon(w,t)$, where
$||\hat \epsilon(w,t)||_\infty\to 0$ as $t\to 0$.
In particular, $\hat \mu$ is invariant by the linear map ${\bf \rho}$, too.
By the change of coordinates $z=K(w)$, the formula~(\ref{linspeedf}) 
reads now:
\begin{equation}\label{linspeedff}
\frac{\rho'(0)}{\rho}=
-\frac{1}{\pi}\int_{\hat C} \frac{\hat\mu(w)}{w^2}d\sigma_w,
\end{equation}
where $\hat C$ is a fundamental region of $w\mapsto \rho w$.
We prove the latter formula. 
From the invariance of $\hat\mu$, one can assume that
$\hat C=\{w: |\rho|<|w|<1\}$.
Differentiating the equation
$\rho(t)\phi_t(w)=\phi_t(\rho w)$
by $t$ at $t=0$, we get, for $w\not=0$:
$$\rho'(0)=\frac{1}{w}(\frac{\partial}{\partial t}|_{t=0}\phi_t(\rho w)-
\rho \frac{\partial}{\partial t}|_{t=0}\phi_t(w)),$$
where, by~(\ref{ahl}),
$$\frac{\partial}{\partial t}|_{t=0}\phi_t(w)=-\frac{1}{\pi}\int_{\C}
\hat\mu(u)\frac{w(w-1)}{u(u-1)(u-w)}d\sigma_u.$$
After elementary transformations and using the invariance
of $\hat\mu$,
we get
$$\rho'(0)=-\frac{\rho(\rho-1)}{\pi}w\sum_{n\in {\bf Z}}
\int_{\hat C}\frac{\hat\mu(\rho^n z)|\rho|^{2n}}{\rho^{n}z
(\rho^n z-\rho w)(\rho^n z-w)}d\sigma_z=$$
$$=-\frac{\rho}{\pi}\int_{\hat C}\frac{\hat\mu(z)}{z}\lim_{N\to +\infty}
\sum_{n=-N}^{N}(\frac{\rho^{n-1}}{\rho^{n-1}z-w}-
\frac{\rho^{n}}{\rho^{n}z-w})
d\sigma_z=$$
$$=-\frac{\rho}{\pi}\int_{\hat C}\frac{\hat\mu(z)}{z}\lim_{N\to +\infty}
(\frac{\rho^{-N-1}}{\rho^{-N-1}z-w}-\frac{\rho^{N}}{\rho^{N}z-w})d\sigma_z=
-\frac{\rho}{\pi}\int_{\hat C}\frac{\hat\mu(z)}{z^2}d\sigma_z,$$
because $|\rho|<1$.
\end{proof}
\paragraph{Adjoint identity.}
We want to integrate the identity
\begin{equation}\label{again}
A(z)-(TA)(z)= \sum_{j=1}^{p}\frac{L_j}{z-v_j}.
\end{equation}
against the $f$-invariant Beltrami form $\mu$. Recall that $\mu$
vanishes in a neighborhood of $\infty$, so that $\mu A=0$ there.
On the other hand, $A$ is not integrable at the points of the
periodic orbit $O$. To deal with this situation, for every small
$r>0$, consider the domain $V_{r}$ to be the plane $\Ci$ with the
following sets deleted: $B(b_1,r)$ union with
$f_{b_{n-k+1}}^{-k}(B(b_1,r))$, for $k=1,...,n-1$, where
$f_{b_{n-k+1}}^{-k}$ is a local branch of $f^{-k}$ taking $b_1\in
O$ to $b_{n-k+1}$. In other words,
$$V_{r}={\Ci}\setminus \{B(b_1,r)\cup_{k=1}^{n-1}
f_{b_{n-k+1}}^{-k}(B(b_1,r))\}.$$ Then $A$ is integrable in
$V_{r}$, and, therefore, by the invariance of $\mu$,
$$\int_{V_{r}}TA(z)\mu(z)d\sigma_z=\int_{f^{-1}(V_{r})}
A(z)\mu(z).$$ Now, $f^{-1}(V_{r})=V_{r}\setminus(C_r\cup
\Delta_r)$, where $C_r=f_{b_1}^{-n}(B(b_1,r))\setminus B(b_1,r)$
is a fundamental region near $b_1$ defined by the local branch
$f_{b_1}^{-n}$ that fixes $b_1$, and, in turn, $\Delta_r$ is an
open set which is away from $O$ and shrinks to a finitely
many points as $r\to 0$. Therefore,
\begin{equation}\label{inta}
\int_{V_{r}}(A(z)-TA(z))\mu(z)d\sigma_z=
\int_{C_r}A(z)\mu(z)d\sigma_z+o_r(1)
\end{equation}
where $o(1)\to 0$ as $r\to 0$. It is easy to see that
$$\int_{C_r}A(z)\mu(z)d\sigma_z=\int_{C_r}\frac{\mu(z)}{(z-b_1)^2}d\sigma_z+o(1).$$
Thus,
\begin{equation}\label{intaex}
\int_{V_{r}}(A(z)-TA(z))\mu(z)d\sigma_z=
\int_{C_r}\frac{\mu(z)}{(z-b_1)^2}d\sigma_z+ o(1).
\end{equation}
The identity~(\ref{again}) then gives us:
\begin{equation}\label{mua}
\int_{C_r}\frac{\mu(z)}{(z-b_1)^2}d\sigma_z+ o(1)=
\sum_{j=1}^{p}L_j\int_{V_{r}}\frac{\mu(z)}{z-v_j}d\sigma_z
\end{equation}
Lemma~\ref{linspeed} allows us to pass to the limit as $r\to 0$:
\begin{equation}\label{adjidentitypol}
\frac{\rho'(0)}{\rho}=
-\frac{1}{\pi}\sum_{j=1}^{p}L_j\int_{\Ci}\frac{\mu(z)}{z-v_j}d\sigma_z.
\end{equation}
\paragraph{Speed of critical values.}
Now we want to express the integral of $\mu(z)/(z-v_j)$ via
$v_j'(0)$. It follows from $v_j(t)=h_t(v_j)$, that
\begin{equation}\label{h}
v_j'(0)=\frac{\partial h_t}{\partial t}|_{t=0}(v_j).
\end{equation}
Let $\psi_t$ be the quasiconformal homeomorphism of
the plane with the complex dilatation $\nu(z,t)$, that fixes $0, 1$
and $\infty$. Since $h_t$ has the same complex dilatation and
fixes $\infty$ too, we have: $h_t=a(t)\psi_t+b(t)$, where $a$,
$b$ are analytic in $t$ (because $h_t$ is so), and $a(0)=1$,
$b(0)=0$, because $h_0=id$. Then
\begin{equation}\label{czspeed}
\frac{\partial h_t}{\partial t}|_{t=0}(z)=a'(0)z+b'(0)+\kappa(z)
\end{equation}
where
\begin{equation}\label{h'}
\kappa(z)=\frac{\partial \psi_t}{\partial t}|_{t=0}(z)=-\frac{1}{\pi}\int_{\Ci}\mu(w)\frac{z(z-1)}{w(w-1)(w-z)}d\sigma_w.
\end{equation}
In other words,
\begin{equation}\label{z'}
\frac{\partial h_t}{\partial t}|_{t=0}(z)=z(a'(0)+
\frac{1}{\pi}\int_{\Ci}\frac{\mu(w)}{w(w-1)}d\sigma_w)
+b'(0)+\frac{1}{\pi}\int_{\Ci}\frac{\mu(w)}{w}d\sigma_w
-\frac{1}{\pi}\int_{\Ci}\frac{\mu(w)}{w-z}d\sigma_w,
\end{equation}
where the integrals exist because
$\mu$ vanishes near $\infty$.

On the other hand, since $f$ stays in the space of centered monic polynomials, we
come to the following normalization at $\infty$:
$h_t(z)=z+O(1/z)$. As $h_t(z)$ is holomorphic in $t$ and $z$ for $|t|$ small
and $|z|$ big, we conclude from this 
that $(\partial h_t/\partial t)|_{t=0}(z)=O(1/z)$.
Coming back to~(\ref{z'}) we see that it is possible if and only if
\begin{equation}\label{z'short}
\frac{\partial h_t}{\partial t}|_{t=0}(z)=
-\frac{1}{\pi}\int_{\Ci}\frac{\mu(w)}{w-z}d\sigma_w.
\end{equation}
In particular,
$$v_j'(0)=-\frac{1}{\pi}\int_{\Ci}\frac{\mu(w)}{w-v_j}d\sigma_w.$$
Plugging this in~(\ref{adjidentitypol}), we get finally:
\begin{equation}\label{polfinal}
\frac{\rho'(0)}{\rho}=\sum_{j=1}^{p}L_j v_j'(0).
\end{equation}
This proves Theorem~\ref{teich}. According to the Concluding
argument of Sect~\ref{red}, Theorem~\ref{teich} yields
Theorem~\ref{main}.

\section{Proof of Theorem~\ref{attr}}\label{attrs}
Assume the contrary: the rank of the matrix ${\bf O}$ is less than
$r$. Then its rows are linearly dependent. Let us write down the
connection~(\ref{ruelleV}) for every periodic orbit $O_j$.
We introduce some notations: $O_j=\{b^j_k\}_{k=1}^{n_j}$ the set of
points of the periodic orbit $O_j$ of period $n_j$,
and the function $\tilde B_j$ is said to be $B_{O_j}$ iff $\rho_j\not=1$
and $\hat B_{O_j}$ iff $\rho_j=1$. Remember that
\begin{equation}\label{bj}
B_{O_j}(z)=\sum_{k=1}^{n_j}
\{\frac{\rho_j}{(z-b^j_k)^2}+\frac{1}{1-\rho_j}\frac{(f^{n_j})"(b^j_k)}{z-b^j_k}\},
\end{equation}
and
\begin{equation}\label{bjn}
\hat B_{O_j}(z)=\sum_{k=1}^{n_j}
\frac{(f^{n_j})"(b^j_k)}{z-b^j_k}.
\end{equation}
First of all, we
observe that each $B_j$ is not identically zero. Indeed, this is
obvious if $\rho_j\not=0$ and $\rho_j\not=1$. But if $\rho_j=0$, then, by the
assumption of the theorem, there is precisely one critical point
$c$ among the points of $O_j$, and $f"(c)\not=0$. One can assume
$b^j_1=c$. Then
$(f^{n_j})"(b^j_1)=f"(c)\Pi_{k=2}^{n_j}f'(b^j_k)\not=0$. This
guarantees that $B_j$ is not zero in this case, too.
If $\rho_j=1$, then, by the assumption, $(f^{n_j})"(b^j_k)\not=0$,
and hence $\hat B_j$ is not zero in this case as well.
In this notation, the connections~(\ref{ruelleV}), (\ref{ruelleVcusp}) 
read as follows:
$\tilde B_j(z)-(T\tilde B_j)(z)=\sum_{i=1}^q 
\frac{\tilde{\partial}^V \rho_j}{\partial V_i}\frac{1}{z-V_i},$
for every $j=1,...,r$.
Now, the assumption implies that there exists a linear
combination $L$ of $\tilde B_j$, $j=1,...,r$, such that
\begin{equation}\label{l}
L(z)-(TL)(z)=0.
\end{equation}
Since no function $\tilde B_j$ is zero and the periodic orbits $O_j$ are
different, $L$ is a non-zero rational function with (possible)
double poles at the points of the periodic orbits $O_j$, $1\le
j\le r$.
Let $j'$ denote indexes corresponding to neutral periodic
orbits $O_j$ (if any). Given $r>0$ small enough, we define a domain
$V_{r,R}$ as the plane with the following sets taken away:
(i) the neighborhood $B^*(R)$ of $\infty$,
(ii) the basin of attraction of $O_j$ provided that $O_j$ is
attractive,
(iii) if $O_j$ is neutral, then the set (to be deleted) is the
disk $B(b_1^j,r)$ union with
$f_{j}^{-k}(B(b^j_1,r))$, for $1\le k\le n_j-1$, where
$f_{j}^{-k}$ is a local branch of $f^{-k}$ taking
$b^j_1\in O_j$ to $b^j_{n_j-k+1}$.

Then $L$ is integrable in $V_{r,R}$. We fix $R$ large enough.
For any fixed parameter $\lambda$, such that $0<\lambda<1$,
and $r>0$ small enough,
\begin{equation}\label{diff}
f^{-1}(V_{r,R})\subset \{V_{r,R}\setminus (f^{-1}(B^*(R))\setminus
B^*(R))\} \cup \cup_{j'} \{B(b^{j'}_1, r)\setminus B(b^{j'}_1,
\lambda r)\}.
\end{equation}
Note that
$\int_{B(b, r)\setminus B(b, \lambda r)}
\frac{1}{|z-b|^2}d\sigma_z\to 2\pi\log\lambda^{-1}$
as $r\to 0$. On the other hand,~(\ref{l}) implies that
$$0=\int_{V_{r,R}}|L-TL|d\sigma_z\ge \int_{V_{r,R}}|L|d\sigma_z
-\int_{V_{r,R}}|TL|d\sigma_z\ge \int_{V_{r,R}}|L|d\sigma_z
-\int_{f^{-1}V_{r,R}}|L|d\sigma_z.$$
As $r\to 0$, we then get
$\int_{f^{-1}(B^*(R))\setminus B^*(R)}|L(z)|d\sigma_z\le
\sum_{j'}C_{j'}\log\lambda^{-1}$, with some $C_{j'}\ge 0$,
which is impossible if $R$ is fixed
and $\lambda$ is close enough to $1$, which contradicts the assumption.
\begin{com}\label{last}
This same proof shows the classical bound $r\le q$ (see
Comment~\ref{fatou}). Indeed, otherwise the rows of ${\bf O}$ are
again linearly dependent, and the proof above applies.
See also Comment~\ref{rlast} for rational functions
and some further discussions.
\end{com}
\section{Rational maps. Main results}\label{rs1}
\subsection{Spaces associated to a rational map}\label{rss1}
Similar to the polynomial case,
let us introduce a space $\Lambda_{d,\bar p'}$ of rational functions
and its subspace
$\Lambda^{q'}_{d, \bar p'}$ as follows.
\begin{defi}\label{ratpol}
Let $d\ge 2$ be an integer, and 
$\bar p'=\{m_j\}_{j=1}^{p'}$ a set of $p'$ positive
integers, such that $\sum_{j=1}^{p'} m_j= 2d-2$. 
A rational
function $f$ of degree $d$ belongs to $\Lambda_{d, \bar p'}$ if and only if 
it satisfies the following conditions:

(1) Infinity is a simple fixed point of $f$; more precisely,
\begin{equation}\label{rrr}
f(z)=\sigma z+m+\frac{P(z)}{Q(z)},
\end{equation}
where $\sigma\not=0,\infty$, and $Q$, $P$ are polynomials of
degrees $d-1$ and at most $d-2$ resp., which have no common roots.
Without loss of generality, one can assume that $Q(z)=z^{d-1}+a_1
z^{d-2}+...+a_{d-1}$ and $P(z)=b_0 z^{d-2}+...+b_{d-2}$,

(2) $f$ has precisely $p'$ geometrically different critical points
$c_1,...,c_{p'}$, and the multiplicity of $c_j$ is equal to $m_j$,
that is, the equation $f(w)=z$ has precisely $m_j+1$ different
solutions for $z$ near $c_j$ and $z\not=c_j$, $j=1,...,p'$.
Denote by $v_1,...,v_{p'}$, $v_j=f(c_j)$, corresponding critical values.
We assume that some of them can coincide as well as 
some can be $\infty$.
By $p=p_f$ we denote usually the number of critical points of $f$ with finite
images, i.e. so that the corresponding critical values are finite.
By definition, $p<p'$ if and only if infinity is a critical value. 

The space $\Lambda_{d, \bar p'}^{q'}$, for some $1\le q'\le p'$, 
consists of those
$f\in \Lambda_{d,\bar p'}$, for which $f$ has precisely $q'$
geometrically different critical values, i.e., the set
$\{v_j=f(c_j), \ j=1,...,p'\}$ contains $q'$ different points
(including possibly infinity). If $\infty$ is a critical value, then
$f$ has $q=q'-1$ different finite critical values.

Finally, we define the space $S_d$ as follows.
Consider first $\Lambda_{d, \overline{2d-2}}$, in other words,
the space of maps with simple critical points. Now, $S_d$ is said to be
its subspace consisting of maps $f$, 
such that every critical value of $f$ is finite.
\end{defi}
By a Mobius change of coordinate, every rational function $f$ 
of degree $d\ge 2$ belongs 
to some $\Lambda_{d, \bar p'}$. Indeed, 
$f$ has either a repelling fixed point,
or a fixed point with the multiplier $1$ (see
e.g.~\cite{mil}). Hence, there exists a Mobius
transformation $M$, such that $\infty$ is a fixed 
non-attracting point of $\tilde f=M\circ f\circ M^{-1}$.

Let us identify $f\in \Delta_{d, \bar p'}$ as above with the point
$$\bar f=\{\sigma, m, a_1,...,a_{d-1}, b_0,...,b_{d-2}\}$$
of ${\Ci}^{2d}$. It defines an analytic (in fact, algebraic)
variety in ${\Ci}^{2d}$. We denote it again by $\Lambda_{d, \bar p'}$.
We will see that $\Lambda_{d, \bar p'}$ has a natural structure of a manifold
of complex dimension $p'+2$, see Sect.~\ref{rlocalcoord}.

The set $\Lambda_{d, \bar p'}$ is connected. 
Apparently, this follows 
from~\cite{conn} although we will {\it not} use this non-trivial statement in the paper.
On the other hand, we will need a much easier fact: the space $S_d$ is
path-connected. This will be used in the proof
precisely like the path-connectedness of the space $\Pi_{d, \bar p}$ is
used in the polynomial case.
To show the path-connectedness of $S_d$, we proceed as follows.
For any two rational functions $f_i(z)=P_i(z)/Q_i(z)$, $i=1,2$,
let us define 
$[f_1, f_2](\gamma):=
((1-\gamma)P_1+tP_2)/((1-\gamma)Q_1+\gamma Q_2)$, for $\gamma\in \Ci$. It
is easy to see that, except for finitely many $\gamma$'s,
$[f_1,f_2](\gamma)\in S_d$ provided $f_1,f_2\in S_d$. Now,
choosing a path $\gamma: [0,1]\to \Ci$ avoiding exceptional
$\gamma$'s, we get a path in $S_d$ that joins their arbitrary two
points $f_1$, $f_2$.
\subsection{Local coordinates}\label{rlocalcoord}
We introduce what is going to be a local coordinate $\bar v(f)$ of
$f$ in $\Lambda_{d, \bar p'}$. 
For a rational function 
$f\in \Lambda_{d, \bar p'}$,
by $c_j(f)$ and $v_j(f)=f(c_j(f))$ we denote its critical points 
and critical values resp., and by $\sigma(f)$, $m(f)$ the 
corresponding data at $\infty$, so that $f(z)=\sigma(f) z + m(f) + O(1/z)$. 
Now, fix $f_0\in \Lambda_{d, \bar p'}$,
and consider maps $f$ in a small enough 
neighborhood of $\bar f_0$ in $\Lambda_{d, \bar p'}$.
Introduce a vector $\bar v(f)\in {\bf C}^{p'+2}$ as follows.
Let us fix an order $c_1(f_0),...,c_{p'}(f_0)$ in the collection 
of all critical points of $f_0$. Moreover, we will do it in such a way,
that:
(a) first $p$ indexes correspond to finite critical values,
i. e. $v_j(f_0)\not=\infty$ for $1\le j\le p$ and $v_j(f_0)=\infty$
for $p<j\le p'$ (if $p<p'$), (b) if $v_i(f_0)=v_j(f_0)$, then
$v_i(f_0)=v_k(f_0)$, for $i\le k\le j$.
There exist $p'$ 
functions $c_1(f),...,c_{p'}(f)$, which are defined and continuous
in a small neighborhood of $f_0$ in $\Lambda_{d, \bar p'}$,
such that they constitute all different critical points of $f$
of the multiplicities $m_j$. 
Define now the vector $\bar v(f)$.
If all critical values of $f_0$ are finite, then 
we set
$$\bar v(f)=\{\sigma(f), m(f), v_1(f),...,v_{p'}(f)\},$$
with the order from above.
If some of the critical values $v_j(f_0)$ of $f_0$ are infinity,
that is, $v_j(f_0)=\infty$ for $p<j\le p'$, then we replace
in the definition of $\bar v(f)$ corresponding $v_j(f)$ by their
reciprocals $v_j(f)^{-1}$:
$$\bar v(f)=\{\sigma(f), m(f), v_1(f),...,v_{p}(f),v_{p+1}(f)^{-1},...,
v_{p'}(f)^{-1}\}.$$
In particular,
$\bar v(f_0)=\{\sigma(f_0), m(f_0), 
v_1(f_0),...,v_p(f_0),0,...,0\}$.

The function $f_0\in \Lambda_{d, \bar p'}$
belongs to a unique
subspace $\Lambda_{d, \bar p'}^{q'}$. Then, 
for $f\in \Lambda_{d, \bar p'}^{q'}$ close to $f_0$, we define
another vector $\bar V(f)\in {\bf C}^{q'+2}$ 
by retaining each critical value in $\bar v(f)$ once.

We have a local map $\delta: f\mapsto \bar v(f)$ from a neighborhood
of $f_0$ in the space $\Lambda_{d, \bar p'}$ to ${\Ci}^{p'+2}$.
\begin{prop}\label{rlocal} 
The map
$$\delta: \Lambda_{d, \bar p'}\to {\Ci}^{p'+2}$$
is locally a biholomorphic isomorphism between some
neighborhoods of $\bar f_0\in \Lambda_{d, \bar p'}$
and $\bar v(f_0)\in \C^{p'+2}$.
In particular, $\Lambda_{d, \bar p'}$ is a manifold of dimension $p'+2$.
\end{prop}
In Sect.~\ref{rpr1} we give two proofs of this basic fact.

The space $\Lambda_{d, \bar p'}$ can therefore be identified in a
neighborhood of its point $f_0$ with a neighborhood
$W_{f_0}$ of $\bar v(f_0)\in {\Ci}^{p'+2}$. If,
moreover, $f_0\in \Lambda_{d, \bar p'}^{q'}$, for some $q'$, then its
neighborhood in $\Lambda_{d, \bar p'}^{q'}$ is the intersection of
$W_{f_0}$ with a $q'$-dimensional linear subspace of
${\Ci}^{p'+2}$ defined by the conditions: $v_i=v_j$ iff
$v_i(f_0)=v_j(f_0)$. Thus the vector $\bar V(f)$ 
serves as a
local coordinate system in $\Lambda_{d, \bar p'}^{q'}$.
\subsection{Connection between the dynamics and parameters:
the main formula}
Let $f$ be a rational function. 
Suppose $f\in \Lambda_{d, \bar p'}$, and, moreover, 
$f\in \Lambda_{d, \bar p'}^{q'}$. These will be the parameter spaces
associated to $f$. 
Consider any finite {\it
periodic orbit} $O=\{b_k\}_{k=1}^n$ of $f$ of exact period $n$,
with the multiplier $\rho=(f^n)'(b_k)=\Pi_{j=1}^n f'(b_j)\not=1$.
(For $\rho=1$, see Subsect.~\ref{rcusp}.)
By the Implicit Function theorem and by
Proposition~\ref{rlocal}, there is a set of $n$ functions
$O(\overline v)=\{b_k(\overline v)\}_{k=1}^n$ defined and
holomorphic in $\overline v\in \Ci^{p'+2}$ in a neighborhood of
$\bar v(f)$, 
such that $O(\overline v)=O$ for $\overline
v=\overline v(f)$, and $O(\bar v)$ is a periodic orbit of $g\in
\Lambda_{d, \bar p}$ of period $n$, where $g$ is in a neighborhood of $f$,
and $\bar v=\bar v(g)$. In particular, if
$\rho(\bar v)$ denotes the multiplier of the periodic orbit
$O(\bar g)$ of $g$, it is a holomorphic function of
$\bar v$ in this neighborhood. The standard notation
$\partial \rho/\partial v_j$, $1\le j\le p$, denotes the partial derivatives
of $\rho$ w.r.t the {\it finite} critical values of $f$.

Now, suppose that $g$ stays in a neighborhood of $f$ inside
$\Lambda_{d, \bar p'}^{q'}$. Set $q=q'$ iff all critical values are finite,
and $q=q'-1$ otherwise.
Then the multiplier $\rho$ of $O(g)$ is, in fact,
a holomorphic function of the vector of $q$ different critical
values $\{V_1,...,V_{q}\}$ of $g$ iff $q'=q$, i.e. they are all finite,
and the vector $\{V_1,...,V_{q},1/V_{q+1}\}$ iff $q'=q+1$,
i.e. $f$ has an infinite critical value.
By $\partial^V\rho/\partial V_k$
we then denote the partial derivatives of $\rho$ w.r.t. the
different {\it finite} critical values $V_1,...,V_q$. 
We have:
$\frac{\partial^V\rho}{\partial V_k}=\sum_{j: v_j=V_k}
\frac{\partial \rho}{\partial v_j}.$
\begin{theo}\label{rmain}
Suppose $f\in \Lambda_{d, \bar p'}$. Let $O$ be a periodic orbit of $f$
with multiplier $\rho\not=1$, and $B=B_O$. Then 
\begin{equation}\label{rruelle}
B(z)-(TB)(z)= \sum_{j=1}^{p}\frac{\partial \rho}{\partial
v_j}\frac{1}{z-v_j},
\end{equation}
where $v_j=v_j(f)$, $1\le l\le p$, 
are all finite critical values corresponding
to different critical points.
We have: 

for $1\le j \le p$ (i.e. for finite critical values of $f$):
\begin{equation}\label{rpartialv}
\frac{\partial \rho}{\partial v_j}=-\frac{1}{(m_j-1)!}
\frac{d^{m_j-1}}{dw^{m_j-1}}|_{w=c_j}\frac{B(w)}{Q_j(w)}=
-\frac{1}{2\pi i}\int_{|w-c_j|=r}\frac{B(w)}{f'(w)}dw,
\end{equation}

for $p<j\le p'$ (i.e. for infinite critical values of $f$):
\begin{equation}\label{rpartialinfv}
\frac{\partial \rho}{\partial (v_j^{-1})}=\frac{1}{(m_j-1)!}
\frac{d^{m_j-1}}{dw^{m_j-1}}|_{w=c_j}\frac{B(w)}{Q_j(w)}=
\frac{1}{2\pi i}\int_{|w-c_j|=r}\frac{B(w)}{(1/f)'(w)}dw,
\end{equation}
where $Q_j$ is defined by $f'(w)=(w-c_j)^{m_j}Q_j(w)$ for
$1\le j\le p$, and  $(1/f)'(w)=(w-c_j)^{m_j}Q_j(w)$ for
$p< j\le p'$.

Also,
\begin{equation}\label{rpartialrest}
\frac{\partial \rho}{\partial
\sigma}=\frac{\tilde\Gamma_2}{\sigma}, \ \ \ \ \ \frac{\partial
\rho}{\partial m}=\frac{\tilde\Gamma_1}{\sigma},
\end{equation}
where $\tilde\Gamma_1$, $\tilde\Gamma_2$ are defined by the
expansion
$$B(z)=\frac{\tilde\Gamma_1}{z}+\frac{\tilde\Gamma_2}{z^2}+O(\frac{1}{z^3})$$
at infinity:
\begin{equation}\label{rexpinf}
\tilde\Gamma_1=\frac{1}{1-\rho}\sum_{k=1}^n (f^n)"(b_k), \ \ \ \
\tilde\Gamma_2= n\rho+\frac{1}{1-\rho}\sum_{k=1}^n b_k(f^n)"(b_k).
\end{equation}

\

If $f\in \Lambda_{d, \bar p'}^{q'}$, then
\begin{equation}\label{rruelleV}
B(z)-(TB)(z)= \sum_{k=1}^{q}\frac{\partial^V \rho}{\partial
V_k}\frac{1}{z-V_k},
\end{equation}
where $V_k$, $k=1,...,q$, are all pairwise different and finite
critical values of $f$.
\end{theo}
The proof is very similar to the one for polynomials, and
is based on the Teichmuller theory of
rational maps. However, it is more technical, because of two
extra parameters $\sigma$, $m$ at $\infty$, 
see Sects.~\ref{rthm}-\ref{infcrv}. 
\subsection{Cusps}\label{rcusp}
Here we consider, very similar to the polynomial case, Subsect.~\ref{cusp}, 
the remaining case $\rho=1$, under the
assumption that the periodic orbit $O=\{b_1,...,b_n\}$ of $f$ is
{\it non-degenerate}: $(f^n)''(b_j)\not=0$,
for some, hence, for any $j=1,...,n$.
Then, for any rational function $g$, which is close to $f$, the map
$g$ in a small neighborhood of $O$
has either precisely two different periodic
orbits $O^\pm_g$ of period $n$ with multipliers
$\rho^\pm\not=1$, or precisely one periodic orbit $O_g$
of period $n$ with the multiplier $1$.

Now, suppose $f\in \Lambda_{d, \bar p'}^{q'}$, 
and let $f_i$, $i=1,2,...$,
be any sequence of maps from $\Lambda_{d, \bar p'}$, such that
$f_i\to f$, $i\to \infty$. We assume that each $f_i$ has a periodic orbit
$O_i$ near $O$, such that
its multiplier $\rho_i\not=1$. In other words, $O_i$
is one of the periodic orbits $O^\pm_{f_i}$.
As in Subsect.~\ref{cusp}, the sequence of functions
$\hat B_i(z)=(1-\rho_i)B_{O_i}(z)$
tends, as $i\to \infty$, to the rational function 
$\hat B(z)=\sum_{b\in O}\frac{(f^n)''(b)}{z-b}$.
As in Subsect.~\ref{cusp}, Theorem~\ref{rmain} implies: 
\begin{prop}\label{rmaincusp}
The following finite limits exist:
\begin{equation}\label{rpartialneutral}
\frac{\hat{\partial} \rho}{\partial v_j}
:=\lim_{i\to \infty}(1-\rho_i)\frac{\partial \rho_i}{\partial v_j}, 
\ \ \ j=1,...,p,
\end{equation} 
\begin{equation}\label{rpartialneutrals}
\frac{\hat{\partial} \rho}{\partial \sigma}
:=\lim_{i\to \infty}(1-\rho_i)\frac{\partial \rho_i}{\partial \sigma}, 
\ \ \ 
\frac{\hat{\partial} \rho}{\partial m}
:=\lim_{i\to \infty}(1-\rho_i)\frac{\partial \rho_i}{\partial m}.
\end{equation}
Then we have:
\begin{equation}\label{rruelleneutral}
\hat B(z)-(T\hat B)(z)= \sum_{j=1}^{p}\frac{\hat{\partial} \rho}{\partial
v_j}\frac{1}{z-v_j}.
\end{equation}
The formula~(\ref{rpartialv}) holds, where one replaces
$\rho$ and $B$ by $\hat\rho$ and $\hat B$ respectively.
Also,
\begin{equation}\label{rpartialneutralrest}
\frac{\hat{\partial} \rho}{\partial
\sigma}=\frac{\hat{\Gamma}_2}{\sigma}, \ \ \ \ \ \frac{\hat{\partial}
\rho}{\partial m}=\frac{\hat{\Gamma}_1}{\sigma},
\end{equation}
where 
$\hat{\Gamma}_1$, 
$\hat{\Gamma}_2$ are defined by the
expansion
$$\hat B(z)=\frac{\hat{\Gamma}_1}{z}+
\frac{\hat{\Gamma}_2}{z^2}+O(\frac{1}{z^3})$$
at infinity:
\begin{equation}\label{rexpinfneutral}
\hat{\Gamma}_1=\sum_{k=1}^n (f^n)"(b_k), \ \ \ \
\hat{\Gamma}_2= \sum_{k=1}^n b_k(f^n)"(b_k).
\end{equation}
Furthermore, if $f$ and $f_i$ are in $\Lambda_{d, \bar p'}^{q'}$, then,
for every $j=1,...,p$, there exists a finite limit
\begin{equation}\label{rpartialneutralV}
\frac{\hat{\partial}^V \rho}{\partial v_j}
:=\lim_{i\to \infty}(1-\rho_i)\frac{\partial^V \rho_i}{\partial v_j},
\end{equation} 
and
\begin{equation}\label{rruelleneutralV}
\hat B(z)-(T\hat B)(z)= \sum_{k=1}^{q}\frac{\hat{\partial}^V \rho}{\partial
V_k}\frac{1}{z-V_k}.
\end{equation}
\end{prop}
\subsection{Multipliers and local coordinates}\label{rlocalsec}
\paragraph{Multipliers versus critical values.}
Theorem~\ref{rmain}, Proposition~\ref{rmaincusp} 
and the contraction property of $T$ yield the
following.
\begin{theo}\label{rattr}
Suppose that $f\in \Lambda_{d, \bar p'}^{q'}$, and
let $V_1,...,V_q$ be all the different and finite critical values
of $f$.
Suppose that $f$
has a collection $O_1,...,O_r$ of
$r$ different finite periodic orbits with the corresponding
multipliers $\rho_1,...,\rho_r$, such that each $O_j$ is
non-repelling: $|\rho_j|\le 1$, $j=1,...,r$. 
Assume that, if, for some $j$, $\rho_j=1$, then the periodic
orbit $O_j$ is non-degenerate. 
Denote by
$\tilde{\partial}^V \rho_j/\partial V_k$ the 
$\partial^V \rho_j/\partial V_k$ iff $\rho_j\not=1$ and
$\hat\partial^V \rho_j/\partial V_k$ iff $\rho_j=1$.
Similar notation stands for $\tilde{\partial}^V \rho_j/\partial \sigma$.

Assume also that if $\rho_j=0$, then the periodic
orbit $O_j$ contains a single critical point, and it is simple.

($H_\infty$). If $\sigma\not=1$ and $m=0$, i.e.,
$f(z)=\sigma z+O(1/z)$ as $z\to \infty$, 
then, for every $1\le k\le q$, such that $V_k\not=0$,
the rank of the following $q\times r$ matrix
\begin{equation}\label{rmdeltacut}
{\bf O}=(\frac{\tilde{\partial} \rho_j}{\partial \sigma}, 
\frac{\tilde{\partial}^V
\rho_j}{\partial V_1},..., \frac{\tilde{\partial}^V \rho_j}{\partial
V_{k-1}}, \frac{\tilde{\partial}^V \rho_j}{\partial V_{k+1}},...,
\frac{\tilde{\partial}^V \rho_j}{\partial V_q})_{1\le j\le r}
\end{equation}
is equal to $r$.

($H_\infty^{attr}$). If $\sigma\not=1$ and $m=0$, and, additionally,
$|\sigma|\ge 1$, i.e.,
$\infty$ is either attracting or neutral fixed point, 
then, for every $1\le k\le q$, such that $V_k\not=0$,
the rank of the following $q-1\times r$ matrix
\begin{equation}\label{rmdeltacuta}
{\bf O^{attr}}=(\frac{\tilde{\partial}^V
\rho_j}{\partial V_1},..., \frac{\tilde{\partial}^V \rho_j}{\partial
V_{k-1}}, \frac{\tilde{\partial}^V \rho_j}{\partial V_{k+1}},...,
\frac{\tilde{\partial}^V \rho_j}{\partial V_q})_{1\le j\le r}
\end{equation}
is equal to $r$.

($NN_\infty$). If $\sigma=1$ and $m\not=0$, then, for every 
$1\le k\le q$, 
the rank of the following $q-1\times r$ matrix
\begin{equation}\label{rmdeltacutn}
{\bf O^{neutral}}=(\frac{\tilde{\partial}^V
\rho_j}{\partial V_1},..., \frac{\tilde{\partial}^V \rho_j}{\partial
V_{k-1}}, \frac{\tilde{\partial}^V \rho_j}{\partial V_{k+1}},...,
\frac{\tilde{\partial}^V \rho_j}{\partial V_q})_{1\le j\le r}
\end{equation}
is equal to $r$.

($ND_\infty$). Finally, if $\sigma=1$ and $m=0$, then, for every 
$1\le k<l\le q$, 
the rank of the following $q-2\times r$ matrix
\begin{equation}\label{rmdeltacutnn}
{\bf O_0^{neutral}}=(\frac{\tilde{\partial}^V
\rho_j}{\partial V_1},..., \frac{\tilde{\partial}^V \rho_j}{\partial
V_{k-1}}, \frac{\tilde{\partial}^V \rho_j}{\partial V_{k+1}},...,
\frac{\tilde{\partial}^V \rho_j}{\partial
V_{l-1}}, \frac{\tilde{\partial}^V \rho_j}{\partial V_{l+1}},...,
\frac{\tilde{\partial}^V \rho_j}{\partial V_q})_{1\le j\le r}
\end{equation}
is equal to $r$.
\end{theo}
For the proof, see Sect.~\ref{rattrs}.
\begin{com}\label{rfatou}
Note that $r\le q$ in the case ($H_\infty$),
$r\le q-1$ in the cases ($H_\infty^{attr}$) and
($NN_\infty$), and $r\le q-2$
in the case ($ND_\infty$) (at least two attracting petals 
at infinity). 
These bounds follow from 
the Fatou-Shishikura inequality, 
see~\cite{shi} and references therein. 
See also Comment~\ref{rlast}. 
\end{com}
\paragraph{Moduli spaces.}
Here we discuss a moduli space of rational
functions associated to a given one with respect to the standard
equivalence relation.
Then we 
apply Theorem~\ref{rattr}.
Most considerations in this paragraph are quite straightforward
consequences of Proposition~\ref{rlocal} and Theorem~\ref{rattr}.

Suppose that $f$ is an arbitrary rational function of degree $d\ge 2$.
Denote by $p'$ and $q'$ respectively 
the number of different critical points and critical values of $f$
in the Riemann sphere, and by $\bar p'$ the vector of multiplicities
at the critical points.
Introduce
a space $Rat^f$ of rational functions $g$ of degree
$d\ge 2$, such that $f$ and $g$ have the same (up to a permutation)
vector $\bar p'$ of multiplicities
at different critical points in the Riemann sphere, and the same number
$q'$ of different critical values. Define the moduli space $Mod^f$ to
be the quotient space $Mod^f=Rat^f/\sim$, 
where $f_1\sim f_2$ iff $f_1,f_2$ are
conjugated by a Mobius transformation.
Denote by $[g]$ the equivalence class of $g\in Rat^f$.
Clearly, $Rat^f$ as well as $Mod^f$ depend merely on $[f]$.
Note that the multiplier of a periodic orbit of $g$
is invariant by a holomorphic conjugation. Therefore, one can speak about
the multiplier of a periodic orbit of the class $[g]$.

The rational function $f$ has a fixed point, which
is either repelling, or has the multiplier $1$ (see e.g.~\cite{mil}). 
Therefore, there is an alternative:
either {\bf (H)} $f$ has a fixed point $a$, such that $f'(a)\not=0,1$,
or {\bf (N)} the multiplier of
every fixed point of $f$ is either $0$ or $1$,
and there is a fixed point with the multiplier $1$. The case {\bf (N)}
is degenerate. We consider each case separately and introduce
a kind of cross-section in the moduli space near $[f]$.

{\bf (H)}. Let $P$ be a Mobius transformation, such that $P(a)=\infty$.
Then $\tilde f=P\circ f\circ P^{-1}$
belongs to $\Lambda_{d, \bar p'}^{q'}$. Moreover,
$P$ can be chosen uniquely in such a way, that one of the critical values of
$\tilde f$ is equal to $1$, and $m(\tilde f)=0$, that is,
$\tilde f(z)=\sigma z+O(1/z)$ at infinity. 
Let us define a submanifold $\Lambda_{\tilde f}$ 
of $\Lambda_{d, \bar p'}^{q'}$
consisting of $g\in \Lambda_{d, \bar p'}^{q'}$
in a neighborhood of $\tilde f$, such that $m(g)=0$, 
and one of the critical values of $g$ is identically equal to $1$.
Introduce the vector 
$$\bar V_{\tilde f}(g)=\{\sigma(g), V_1(g),...,V_{q'-2}(g), V_{q'-1}^*(g)\},$$
such that $V_1(g),...,V_{q'-1}(g), 1$ are all different
critical values of $g$,
and $V_1(g),...,V_{q'-2}(g)$ are finite
while $V_{q'-1}^*(g)=V_{q'-1}(g)$ iff $V_{q'-1}(\tilde f)$ is finite
and $V_{q'-1}^*(g)=1/V_{q'-1}(g)$ otherwise. 
We denote $q=q'$ in the former case, and $q=q'-1$ in the latter one.
We see from Proposition~\ref{rlocal}, that $\bar V_{\tilde f}$
is a local coordinate of $\Lambda_{\tilde f}$: the correspondence
$g\in \Lambda_{\tilde f}\mapsto \bar V_{\tilde f}(g)\in {\bf C}^{q'}$
is biholomorphic from the manifold $\Lambda_{\tilde f}$ onto a neighborhood
of the point $\bar V_{\tilde f}(\tilde f)$ in ${\bf C}^{q'}$.
Now we have a natural projection $[.]_V: \bar V\mapsto [g]$ from 
a neighborhood of $\bar V_{\tilde f}(\tilde f)\in {\bf C^{q'}}$
into the space $Mod^{\tilde f}$, where $[\bar V]_V$ is said to be
the equivalence class of the
unique $g\in \Lambda_{\tilde f}$, such that $\bar V_{\tilde f}(g)=\bar V$.

{\bf (N)}. There are two subcases to distinguish.

{\bf (NN)}: $f$ has a fixed point $a$, such that 
$f'(a)=1$ and $f''(a)\not=0$. 
Let $P$ be a Mobius transformation, such that $P(a)=\infty$.
Then $\tilde f=P\circ f\circ P^{-1}$
belongs to $\Lambda_{d, \bar p'}^{q'}$. Moreover,
$P$ can be chosen uniquely in such a way, that one of the critical values of
$\tilde f$ is equal to $1$, and $m(\tilde f)=1$, that is,
$\tilde f(z)=z+1+O(1/z)$ at infinity. 
Then we define $\Lambda_{\tilde f}$ to be the set of all
$g\in \Lambda_{d, \bar p'}^{q'}$
in a neighborhood of $\tilde f$, such that $m(g)=1$, 
and the critical value $V_{q'}(g)$
of $g$ (which is close to $V_{q'}(\tilde f)=1$) is identically equal to $1$.
The vector $\bar V_{\tilde f}$ is defined like in the previous case.
It is a coordinate in $\Lambda_{\tilde f}$. 
As above, there is the projection
$[.]_V: \bar V\mapsto [g]$ from 
a neighborhood of $\bar V_{\tilde f}(\tilde f)$ in ${\bf C^{q'}}$
into the space $Mod^{\tilde f}$.

{\bf (ND)}: every fixed point with multiplier $1$
is degenerate. Let $a$ be one of them: $f'(a)=1$ and
$f''(a)=0$. Then the Mobius map $P$ can be chosen
uniquely in such a way, that  
$\tilde f(z)=P\circ f\circ P^{-1}(z)=
z+O(1/z)$, and $\tilde f$ has a critical value equal to $1$
in one attracting petal of $\infty$, and equal to $0$
in another attracting petal of $\infty$. 
Then $\Lambda_{\tilde f}$ consists of $g\in \Lambda_{d, \bar p'}^{q'}$
in a neighborhood of $\tilde f$, such that the 
critical value of $g$, which is close to $V_{q'-1}(\tilde f)=1$ 
is identically equal to
$1$, and the critical value of $g$, which is close to $V_{q'}(\tilde f)=0$,
is identically equal to $0$. 
Define
$$\bar V_{\tilde f}(g)=
\{\sigma(g), m(g), V_1(g),...,V_{q'-3}(g), V_{q'-2}^*(g)\},$$
such that $V_1(g),...,V_{q'-2}(g), 1, 0$ are all different
critical values of $g$,
and $V_1(g),...,V_{q'-3}(g)$ are finite
while $V_{q'-2}^*(g)=V_{q'-2}(g)$ iff $V_{q'-2}(\tilde f)$ is finite
and $V_{q'-2}^*(g)=1/V_{q'-2}(g)$ otherwise. 
We denote $q=q'-1$ in the former case, and $q=q'-2$ in the latter one.
We see from Proposition~\ref{rlocal}, that $\bar V_{\tilde f}$
is a local coordinate on $\Lambda_{\tilde f}$. 
There is the projection $[.]_V: \bar V\mapsto [g]$. 

\

It is not hard to understand that the map $[.]_V$ sends a small 
neighborhood
of the point $\bar V_{\tilde f}(\tilde f)$ in ${\bf C^{q'}}$
onto a neighborhood of the point $[\tilde f]=[f]$ in $Mod^f$.
In fact, the map $[.]_V$ defines a complex $q'$-orbifold structure
on $Mod^f$ (see e.g.~\cite{McM} for the definition of an orbifold).

Let us reformulate Theorem~\ref{rattr} for the map $[.]_V$.
Suppose that $[f]$ has a collection of
$r$ different {\it non-repelling} periodic orbits 
with 
multipliers $\rho^0_1,...,\rho^0_r$, i.e. $|\rho^0_j|\le 1$. 
Assume additionally that
$\rho^0_j\not=1$, $j=1,...,r$, and, if $\rho^0_j=0$, for some $j$, 
then the corresponding periodic
orbit contains a single and simple critical point of the map.
Let us consider a map $\tilde f$ corresponding to $f$, and 
fix an order $O_1,...,O_r$ of the above periodic orbits of $\tilde f$. 
We then
have a vector $(\rho^0_1,...,\rho^0_r)$ of their multipliers.
If the map changes, the multipliers become functions $\rho_j(g)$ 
of the map $g$. In particular, $\rho_j(\tilde f)=\rho_j^0$.
Then Theorem~\ref{rattr} implies the following:
\begin{theo}\label{rattrmod}
There are $r$ indexes $1\le j_1\le...\le j_r\le q'$
as follows.
One can replace the map $[.]_V$ 
by another map $[.]_\rho: \bar V_\rho\to Mod^{\tilde f}$ 
defined in a neighborhood
of a point $\bar V_\rho(\tilde f)\in {\bf C}^{q'}$,
where $\bar V_\rho$ is obtained from $\bar V_{\tilde f}$
by replacing the coordinates with indexes $j_1,...,j_r$
by $\rho_{j_1},...,\rho_{j_r}$ respectively.
The change of variables is biholomorphic.
Moreover, in the case (H), if $|\sigma(\tilde f)|\ge 1$, 
and in the case (NN) the above $r$ coordinates
in $\bar V_{\tilde f}$ can be chosen among the critical values 
$V_1,...,V_{q'-1}$, while in the case (ND) they can be chosen
among the critical values $V_1,...,V_{q'-2}$.
\end{theo}
Let us consider a particular case of Theorem~\ref{rattrmod}, which 
corresponds to 
maps with the maximal number of non-repelling periodic orbits.
Namely, assume that the number of such orbits is $r=q=q'$.
Then $[f]$ has necessarily a repelling fixed point, i.e. we are in the case
(H). Therefore, the map $[.]_\rho$
depends on $\rho_1,...,\rho_q$ only.
It has an invariance property as follows.
Suppose that
$[\rho_1,...\rho_q]_\rho=[\rho_1',...,\rho_q']_\rho$, 
for two vectors of the
multipliers $(\rho_1,...\rho_q)$, $(\rho_1',...\rho_q')$, which 
correspond to
maps $g, g'$ respectively. It is clear that then $g$ and $g'$
are conjugate by a Mobius transformation
$M$. In turn, it defines a permutation $\pi$
of $1,...,q$ in the collection
$O_1,...,O_q$: $M$ maps a periodic orbit $O_j(g)$ of $g$ to
the periodic orbit $O_{\pi(j)}(g')$ of $g'$. Then the invariance property is:
$\rho_{\pi(j)}'=\rho_j$, $j=1,...,q$. 
It has an interesting consequence: 

{\it In the above set up $r=q=q'$, i.e., if the number of non-repelling
periodic orbits with the multipliers different from $1$ is equal to the number
of different critical values, the map
$[.]_\rho$ is locally injective in each coordinate $\rho_j$, $j=1,...,q$.}

Indeed, if $[\rho_1,\rho_2,...\rho_q]_\rho=[\rho_1',\rho_2,...\rho_q]_\rho$,
then there is a permutation $\pi$ as above.
We have: if $\pi(1)=1$, then $\rho_1'=\rho_1$, and otherwise
$\rho_1=\rho_{\pi(1)}=\rho_{\pi^2(1)}=...=\rho_{\pi^{l-1}(1)}$,
where $l\ge 1$ is minimal so that $\pi^l(1)=1$. But $\pi(\pi^{l-1}(1))=1$,
hence, $\rho_1'=\rho_{\pi^{l-1}(1)}=\rho_1$.
 
\

Let us find the degree of the map $[.]_V$. 
We fix the manifold $\Lambda_{\tilde f}$ in a small enough
neighborhood of $\tilde f$ in such a way, that it is projected
onto a neighborhood of $[f]$ in $Mod^f$.
Denote by $|A|$ the number of points in a set $A$
(a priori, $|A|$ could be infinite).  
Given a rational function $g$, we denote by $Aut(g)$ the 
finite group of Mobius transformations that commute with $g$.
For $g\in \Lambda_{\tilde f}$, denote by ${\bf M(g)}$ the set of Mobius 
transformations $M$, such that  $M^{-1}\circ g\circ M\in \Lambda^{\tilde f}$.
Obviously, if $M\in {\bf M(g)}$ and $K\in Aut(g)$, then
$K\circ M\in {\bf M(g)}$. It follows, $|Aut(g)|$ divides $|{\bf M(g)}|$.
The quantity $|{\bf M(g)}|/|Aut(g)|$ is precisely
the number of different maps $\psi\in \Lambda^{\tilde f}$, such that
$\psi\in [g]$.

{\bf Claim.} {\it Let $g\in \Lambda_{\tilde f}$.
Then $|{\bf M(g)}|$  
is finite and equal to 
$$|{\bf M(g)}|=D_g:=\sum_{P\in Aut(\tilde f)} L_g(P(\infty)),$$
where $L_g(Z)$ is the number of geometrically
different fixed points of $g$ near a fixed point $Z$ of $\tilde f$. 
Consequently, 
the number of different maps $\psi\in \Lambda_{\tilde f}$, such that
$\psi\in [g]$, is equal to $\frac{D_g}{|Aut(g)|}$.
}

We will not prove it here (and will not use it in the paper)
although the consideration behind the proof is quite clear.
Namely, for every $P\in Aut(\tilde f)$ and every fixed point 
$Z_g$ of $g$ which is near $P(\infty)$
there is one and only one $M\in {\bf M(g)}$, such that $M(\infty)=Z_g$
and $M$ is close to $P$.

We have the bound: $|{\bf M(g)}|\le |Aut(\tilde f)| L_m$, where
$L_m$ is the maximal number of fixed points that can
appear from the fixed point $\infty$ of the map 
$\tilde f$ under a perturbation of the map. Let us be more precise.
For every $P\in Aut(\tilde f)$, the point $P(\infty)$
is a fixed point of $\tilde f$ with the same multiplier as at $\infty$.
We have: $1\le L_g(P(\infty))\le L_m$, where $L_m=1$ in the case 
$\sigma(\tilde f)\not=1$, and, if $\sigma(\tilde f)=1$, the $L_m\ge 2$
is defined by:
$\tilde f(z)=z+b/z^{L_m-2}+...$, $z\to \infty$, with $b\not=0$.
Let us call the number $L_m$ the 
multiplicity of the fixed point. It is defined similarly for
any fixed point with the multiplier $1$.

Let us discuss briefly several cases.
Assume first that $\sigma(\tilde f)\not=1$.
Then $|{\bf M(g)}|=|Aut(\tilde f)|$ 
(hence, independent of $g\in \Lambda^{\tilde f}$).
If, additionally, $Aut(f)=\{I\}$, then 
$[.]_V$ is injective. 
On the other hand, if
$\sigma(\tilde f)=1$ , let us assume that $g$ is non-degenerate, 
in a sense, that, firstly, $Aut(g)$ is trivial, and, secondly,
$L_g(P(\infty))=L_m$, i.e. $g$ has the maximal number of fixed points
near $P(\infty)$, for every $P\in Aut(\tilde f)$.
Then the number of different maps $\psi\in \Lambda_{\tilde f}$, such that
$\psi\in [g]$, is maximal and equal to $|{\bf M(g)}|=|Aut(\tilde f)| L_m$. 

Let us come back to the general case.
As $g$ changes and different $\psi\in [g]$ from $\Lambda_{\tilde f}$
collide, this corresponds either
to a collision of fixed points of $g$ near some $P(\infty)$
or to the appearence of new maps
in the group $Aut(g)$. To be more specific, given 
$g\in \Lambda_{\tilde f}$, define
an equivalence relation in the set of all fixed points
of $g$ near the set $\bar Z=\{P(\infty)\}_{P\in Aut(\tilde f)}$
as follows: two points $x, y\in \bar Z$ are equivalent if and only if there exists
$K\in Aut(g)$ so that $y=K(x)$. To every equivalence class in $\bar Z$
there corresponds one and only one map $\psi\in [g]$, such that
$\psi\in \Lambda_{\tilde f}$.
We define the multiplicity of $\psi$ as the sum of the
multiplicities of all fixed points of $g$ in this equivalence class
(note that the multiplicities of all fixed points of the same class
are equal).
With this definition, we have (without any restriction on $\tilde f$): 
for every $g\in \Lambda_{\tilde f}$, the total number of 
$\psi\in [g]$ in $\Lambda_{\tilde f}$
each counted with its
multiplicity is equal to $|Aut(\tilde f)| L_m$. 
\paragraph{Quadratic rational maps.}\label{quad}
For a degree two rational function $f$, $\bar p'=(1,1)$ and $q'=2$, hence,
$Rat_2=Rat^f$ is the set of all quadratic rational maps,
and $Mod_2=Mod^f$ is the space of orbits of the quadratic maps
by Mobius conjugations. It is easy to check that
the degree of the map $[.]_V$ takes values $1, 2$, or $6$. 
The spaces $Rat_2$ and $Mod_2$
have been studied intensively, see~\cite{mary},~\cite{mary1},
~\cite{maryaster},~\cite{Mild2}.
Global coordinates in $Mod_2$ are introduced in~\cite{Mild2}.
It turnes out $Mod_2$ is isomorphic to ${\bf C}^2$.
The problem of multipliers as coordinates 
for hyperbolic (and some neutral) degree $2$ rational maps 
is settled in~\cite{mary}. 
Theorem~\ref{rattr} allows us 
to deal with not necessary hyperbolic maps.
Let us state its corollary for degree two.

Suppose $f$ is a rational function of degree $2$ that has a 
periodic orbit $O$ with the multiplier $\rho$, such that $|\rho|\le 1$.
If $\rho=0$, assume that the orbit $O$ contains a single critical point.
If $\rho=1$, assume that $O$ is not degenerate (i.e. each point of $O$
has only one attracting petal) and, moreover, if $f$ is  
conjugate to $z^2+1/4$, then $O$ is not its neutral fixed point. 
Then, after a Mobius change of coordinates, 
$f(z)=\sigma z+m+O(1/z)$ and $O\not=\infty$,
and also one of the (two) different critical values $v_1, v_2$
of $f$, say, $v_2$ is not zero. 
Moreover, if $\sigma\not=1$, one can further assume 
that $m=0$. As usual, the multiplier $\rho$ is a function
of $\sigma, m, v_1, v_2$
(for the moduli space, one can keep $m$ and $v_2$ fixed though,
see the general discussion above). 
\begin{coro}\label{corquad}
For $\rho\not=1$, the vector 
$(\partial \rho/\partial \sigma, \partial \rho/\partial v_1)$ is not zero,
and for $\rho=1$, the vector 
$(\partial \hat\rho/\partial \sigma, \partial \hat\rho/\partial v_1)$ 
is not zero.
Moreover, under the condition $|\sigma|\ge 1$
(i.e., the fixed point at $\infty$ is not repelling), we have:  
$$\partial \rho/\partial v_1\not=0$$ 
for $\rho\not=1$, and 
$$\partial \hat\rho/\partial v_1\not=0$$ 
for $\rho=1$.
\end{coro}
\section{Theorem~\ref{rmain}: an outline of the proof}\label{rthm}
First, we prove Theorem~\ref{rmain} for maps $f$ in the space $S_d$, that is,
assuming that every critical
point of $f$ is simple and every critical values is finite. 
It will occupy most of the rest of the paper.
Then we prove Theorem~\ref{rmain} for multiple critical points
still assuming that every critical value is finite.
For this, we use a kind of a limit procedure, see Sect.~\ref{multiple}.
Finally, to complete the 
proof of Theorem~\ref{rmain}, we send
some of the critical values to $\infty$, see Sect.~\ref{infcrv}.
So, assume (until Sect.~\ref{multiple}) that $f\in S_d$. 
Since $\rho$ is a holomorphic function in $\bar v$,
it is enough to prove the
formulae of Theorem~\ref{rmain} for $\rho\not=0$.
\paragraph{The identity.} We will use the
same identity of Theorem~\ref{rformula}.
\paragraph{Reduction to the hyperbolic case.}
Here we show that in order to prove Thoerem~\ref{rmain} for any $f\in S_d$,
it is enough to prove it only
for those $f$ from $S_d$ that satisfy the following
conditions:

(1) $f$ is a hyperbolic map,  and $\infty$ is an attracting fixed
point, i.e., $|\sigma|>1$,

(2) $f$ has no critical relations,

(3) $O$ is an attracting periodic orbit of $f$.

Indeed, assume that Theorem~\ref{rmain} holds for this open subset
of maps from $S_d$. Given now any $f\in S_d$ as in
Theorem~\ref{rmain}, we find a real analytic path $g_t$, $t\in
[0,1]$, in $S_d$, which has the following
properties:
(i) $g_0=f$,
(ii) $g_1$ satisfies conditions (1)-(3),
(iii) the analytic continuation $O_t$ (a periodic orbit of $g_t$)
of the periodic orbit $O$ along the path is well-defined (i.e. the
multiplier of $O_t$ is not $1$ for $t\in [0,1]$), and $O_1$ (the
periodic orbit of $g_1$) is attracting.

Denote by $\Delta(z, \overline f)$ the difference between the left
and the right hand sides of~(\ref{rruelle}). It is an analytic
function in $\overline f$ in a neighborhood of every point
$\bar{g_t}$, $t\in [0,1]$. On the other hand, by the assumption,
it is identically zero in a neighborhood of $\bar{g_1}$. By the
Uniqueness Theorem for analytic functions, $\Delta(z, \bar f)=0$.

Let us show that the path $g_t$ as above exists.
We first connect $f$ to the map $p_0: z\mapsto z^d$
through a path $\gamma_0$ of the form $[f, p_0]$ (see Subsect.~\ref{rss1}), 
so that the
analytic continuation of the periodic orbit $O$ of $f$ along this
path exists, and $O$ turns into a periodic orbit $Q$ of $p_0$.
Then we proceed by a real analytic path $c_Q$ in the parameter
plane of $p_c(z)=z^d+c$ that turns $Q$ into an attracting periodic orbit of
some $p_c$. Finally, we find the desired path $g_t$ in $S_d$
in a neighborhood of $c_Q\circ \gamma_0$.
\paragraph{Hyperbolic maps}
Here we describe how to prove Theorem~\ref{rmain} for the
maps {\it $f\in S_d$ that satisfy the conditions (1)-(3) of the previous
paragraph.}
Similar to the polynomial case, let $\nu(z, t)$ be an
analytic family of invariant Beltrami coefficients in the
Riemann sphere, and $\nu(z,0)=0$. (We do not assume that $\nu(z,t)=0$
for $z$ near $\infty$ though.)
In turn, let
$h_t$ be an analytic family of quasiconformal homeomorphism in the
plane that fix $\infty$, so that $h_t$ has the complex dilatation
$\nu(z,t)$, and $h_0=id$. Then $f_t=h_t\circ f\circ h_t^{-1}$
is an analytic family
of rational functions. Moreover, $f_t\in
S_d$, and $O_t=h_t(O)$ the corresponding
attracting periodic orbit of $f_t$. Denote by $\rho(t)$ its
multiplier, and by $v_j(t)=h_t(v_j)$ the set of finite critical
values of $f_t$. Besides,
$f_t(z)=\sigma(t)z+m(t)z+O(\frac{1}{z})$.
Note that the functions $\sigma(t)$, $m(t)$, $\rho(t)$, and
$v_j(t)$ are analytic in $t$, and $\sigma(0)=\sigma$, $m(0)=m$,
$\rho(0)=\rho$, $v_j(0)=v_j$.
Starting with Theorem~\ref{rformula}, we derive:
\begin{theo}\label{rteich}
For $f\in S_d$,
\begin{equation}\label{rteichform}
\frac{\rho'(0)}{\rho}=\Gamma_2\frac{\sigma'(0)}{\sigma}+
\frac{\Gamma_1}{\sigma}m'(0)+\sum_{j=1}^{2d-2} L_j v_j'(0),
\end{equation}
where $\Gamma_1$ and $\Gamma_2$ are defined by the expansion at infinity:
$$A(z)=\frac{\Gamma_1}{z}+\frac{\Gamma_2}{z^2}+O(\frac{1}{z^3}).$$
\end{theo}
In the course of the proof we calculate $\sigma'(0)$ and $m'(0)$.
\paragraph{Concluding argument.}
We are going to compare~(\ref{rteichform}) to the following
obvious identity:
\begin{equation}\label{rteichformal}
\rho'(0)=\frac{\partial\rho}{\partial\sigma}\sigma'(0)+
\frac{\partial\rho}{\partial m}m'(0)+\sum_{j=1}^{2d-2}
\frac{\partial\rho}{\partial v_j}v_j'(0).
\end{equation}

The proof will be finished once we will show that the vector
$$\{\sigma'(0), m'(0), v_1'(0),...,v_{2d-2}'(0)\}$$ 
can be taken arbitrary in ${\bf C}^{2d}$.
To this end,
for every vector $\bar v'=\{\sigma', m', v_1',...,v_{2d-2}'\}\in
{\Ci}^{2d}$ of initial conditions there exists an analytic family
$f_t$ of rational maps from $S_d$ with the critical values
$v_1(t),...,v_{2d-2}(t)$, such that
$f_t(z)=\sigma(t)z+m(t)+O(\frac{1}{z})$,
and $\sigma'(0)=\sigma'$, $m'(0)=m'$, $v_j'(0)=v_j'$, for $1\le
j\le 2d-2$. Indeed, this is an immediate consequence of
Proposition~\ref{rlocal} for $f\in S_d$, 
where one can simply take locally $\bar
v(t)=\bar v+t \bar v'$, and find by the inverse holomorphic
correspondence $\bar v\mapsto \bar f$ the corresponding local
family $f_t$, such that $f_0=f$. Since $f$ is hyperbolic and has
no critical relations, the following fundamental facts hold: every
nearby map $f_t$ is conjugate to $f$ by a quasiconformal
homeomorphism $h_t$, and $h_t$ can be chosen to be analytic in
$t$. Furthermore, the complex dilatations of $h_t$ form a family
$\nu(z,t)$ as described above. For $f\in
S_d$ and without critical relations, 
this is an immediate corollary of ~\cite{mcmsul},
Theorem 7.4, and~\cite{BR}, Theorem 3.
This shows that the vector $\{\sigma'(0), m'(0),
v_1'(0),...,v_{2d-2}'(0)\}$ can be chosen arbitrary, 
and, hence, proves that Theorem~\ref{rteich}
implies Theorem~\ref{rmain}.
\section{Proof of Proposition~\ref{rlocal} }\label{rpr1}
We present two proofs of this basic fact.
The first proof uses general properties of analytic sets,
and it is very similar to the proof
of Proposition~\ref{local}.
The second one is a direct and nice
construction of the (local) inverse map $\delta^{-1}$
with help of quasiconformal surgery. 
We use an idea by Eremenko and follow essentially~\cite{er},
where it is done for polynomials and for a single critical value.
It gives an alternative proof of Proposition~\ref{local} as well. 

Both proofs start as follows.
Denote $\Lambda=\Lambda_{d, \bar p'}$. 
The map $f$ has a critical point $c$ of multiplicity $m\ge 1$ 
with a finite critical value $v=f(c)$ if and only if
$c$ satisfies the following conditions:
$f'(c)=0,...,f^{(m)}(c)=0, \ \ \ f^{(m+1)}(c)\not=0$.
From the latter two conditions, one can express $c$ as
a local holomorphic function $c=\phi_m(\bar f)$ of the vector
of the coefficients
$\bar f\in {\bf C}^{2d}$. This determines
$m-1$ algebraic equations $\Psi_{k,m}(\bar f)=0$,
where $\Psi_{k,m}(\bar f)=f^{(k)}(\phi_m(\bar f))$, $k=1,...,m-1$.
If the critical value $v$ is infinite, the conclusion
is the same (considering $1/f$), and we will use similar notations
in this case as well.
Denote $\bar \Psi(\bar f)=
\{\Psi_{k,m_j}(\bar f)=0\}_{j=1,k=1}^{j=p',k=m_j}$.
Thus the analytic set $\Lambda$ in ${\bf C}^{2d}$
is determined by the following $2d-2-p'$ equations
of the vector $\bar f\in {\bf C}^{2d}$:
$\bar \Psi_{k,m_j}(\bar f)=0$.

Secondly, we have the map $\delta: \Lambda\to {\bf C}^{p'+2}$
defined by $\delta(\bar f)=\bar v$. It can be represented
as the restriction on $\Lambda$
of the following map (denoted by $\tilde\delta$),
which is (locally) holomorphic in $\bar f\in {\bf C}^{2d}$:
$$\tilde\delta(\bar f)=\{\sigma(\bar f), m(\bar f), f(\phi_{m_1}(\bar f)),...,
f(\phi_{m_p}(\bar f)), 1/f(\phi_{m_{p+1}}(\bar f)),...,
1/f(\phi_{m_{p'}}(\bar f))\}.$$
As $\delta:\Lambda\to {\bf C}^{p'+2}$ has a holomorphic extension 
$\tilde\delta$,
it is enough to prove the following claim: 
the map $\delta:\Lambda\to {\bf C}^{p'+2}$
maps a neighborhood in $\Lambda$ of every $\bar f_0\in \Lambda$
onto a neighborhood in ${\bf C}^{p'+2}$ of the point $\delta(\bar f_0)$ and
has a local holomorphic inverse $\delta^{-1}$.
We present two proofs of this claim.
\subsection{First proof}\label{first}
It is similar to the proof of Proposition~\ref{local}, see Sect.~\ref{pr1}.
The following lemma is crucial:
\begin{lem}\label{rinj}
The map $\delta:\Lambda\to {\bf C}^{p'+2}$
is injective in a neighborhood of every $\bar f_0\in \Lambda$.
\end{lem}
The proof of this Lemma is almost identical to the proof of the injectivity
of the map $\pi$ in
Proposition~\ref{local}, so we omit it.

Next, we use the following well-known statement about analytic sets.
Its particular case (for $r=l$) was used to prove Proposition~\ref{local}. 
\begin{prop}\label{verygeneral}
Let $U$ be a ball in $\C^l$,
and let $E$ be an analytic set in $U$, which is defined
as the set of common zeros of $l-r$ holomorphic functions in $U$, 
for some $0<r\le l$. Assume $g: U\to \C^r$ is a holomorphic
map, which is injective on $E$.
Then $g(E)$ is an open set in $\C^r$ and $g: E\to \C^r$ has a holomorphic inverse
on this set.
\end{prop}
Proposition~\ref{rlocal} follows immediately from Lemma~\ref{rinj}
and Proposition~\ref{verygeneral}, if we set $l=2d$, $r=p'+2$,
$g=\tilde\delta$, $U$ to be a ball around $\bar f_0\in \Lambda$, 
where $\delta$ is injective on $\Lambda$,
and $E=\Lambda\cap U$.

It remains to prove the above Proposition~\ref{verygeneral}.
Consider the restriction $g|_E$ of the holomorphic map $g: U\to \C^r$
on $E$. Since $g$ is injective on $E$, every point $z_0\in E$ is obviously
an isolated point in the set $g|_E^{-1}(g(z_0))$. Therefore (see e.g.~\cite{whit},
Chapter 4, Theorem 6B), for some neighborhood $W$ of $z_0$, the set
$F:=g(E\cap W)$ is analytic, and the dimension of $F$ at the point $g(z_0)$
is equal to the dimension of $E$ at $z_0$. 
On the other hand, the dimension of $E$ at each point is at least $r$
because $E$ is defined in $\C^l$ by $l-r$ equations
(\cite{whit}, Chapter 2, Theorem 12G).
Hence, as the analytic set $F$ lies in $\C^r$ and its dimension is at least $r$,
it is equal to $r$ and $F$ is a neighborhood of $g(z_0)$ in $\C^r$. 
Thus $g(E)$ is open in $\C^r$.
Now, the map $g|_E^{-1}$
is well-defined on this open set, and it is analytic
in a neighborhood of the image of every regular point of $E$.
On the rest of the points, which form an analytic set
of smaller dimension, $g|_E^{-1}$ is locally bounded. By the extended
Riemann removable singularity theorem, 
$g|_E^{-1}$ is holomorphic everywhere.
\subsection{Second proof}\label{erpr}
Let $f_0\in \Lambda_{d, \bar p'}$ and 
$\bar v(f_0)=\delta(\bar f_0)=\{\sigma_0,m_0,v_1^0,...,v_p^0,0,...,0\}$.
We prove the existence of a local holomorphic inverse $\delta^{-1}$
by constructing a rational function $f\in \Lambda_{d, \bar p'}$
with a prescribed $\bar v=\bar v(f)$ so that $\bar v$ 
is close to $\bar v(f_0)$
and $\bar f$ depends holomorphically on $\bar v$.
To this end, choose small (in the spherical metric) pairwise 
disjoint disks $B_k$, $k=1,...,q_0$, centered at the critical
values of $f_0$. Let $D_j$, $j=1,...,p'$, be the components
of $f_0$-preimages of all $B_k$ on which $f_0$ is not one-to-one. 
Each $D_j$ is small
(in the Euclidean metric) and contains one and only one
critical point $c^0_j$ of $f_0$. Given any vector
$\bar v=\{\sigma, m, v_1,...,v_p,v^*_{p+1},...,v^*_{p'}\}$
close to $\bar v(f_0)$, and given $1\le j\le p'$,
one can choose a diffeomorphism $\phi_j$
of the Riemann sphere, which satisfies the following
conditions: 
(1) $\phi_j(z)$ depends on $z$ and $v_j$ only,
and $\phi_j$ is the identity outside of the ball $B_{k(j)}=f_0(D_j)$,
(2) $\phi_j(v^0_j)=v_j$, if $j=1,...,p$, and
$\phi_j(\infty)=1/v^*_j$, if $j=p+1,...,p'$, and (3)
$\phi_j$ depends holomorphically on $v_j$, if
$j=1,...,p$, and on $v^*_j$, if $j=p+1,...,p'$.
Such $\phi_j$ can be constructed, for example, as in~\cite{er}.
First, for a disk $B=B(a,r)$, 
set $\chi_B(z)$ to be $0$, if $z\notin B$,
and $\chi_B(z)=(1-|z-a|^2/r^2)^2$, if $z\in B$. Define
$\phi_{B, b}(z)=z+(b-a)\chi_B(z)$. If $|b-a|$ is small enough, then
$\phi_{B, b}$ is a diffemorphism of $\C$, such that $\phi_{B, b}(a)=b$.
For a disk $B=B^*(R)$ around $\infty$, we denote $B_0=B(0, 1/R)$
and set
$\phi_{B, b}=J\circ \phi_{B_0, 1/b}\circ J^{-1}$, where $J(z)=1/z$.
Now we can define: $\phi_j=\phi_{B_k(j), v_j}$ for $1\le j\le p$,
and $\phi_j=\phi_{B_\infty, 1/v_j^*}$ for $p+1\le j\le p'$,
where $B_\infty$ is the disk centered at $\infty$, which is among $B_k$
(provided $p<p'$).

Now, define a new function $f^*$, such that
$f^*(z)=f_0(z)$ outside of all $D_j$ and
$f^*(z)=\phi_j(f_0(z))$ if $z\in D_j$. Note that $f^*=f_0$ in a definite
neighborhood of $\infty$.
Also, $f^*(z)$ depends holomorphically on $\bar v$ for every $z$,
and $f^*\to f_0$, as $\bar v\to \bar v(f_0)$, uniformly on the Riemann sphere.
The map $f^*$ is a degree $d$ smooth map of the Riemann sphere
with the critical values at $v_1,...,v_p,1/v^*_{p+1},...,1/v^*_{p'}$
and with the same expansion at $\infty$: $f^*(z)=\sigma_0 z+m_0+O(1/z)$.
Let 
$\mu=\frac{\partial f^*}{\partial \bar z}/\frac{\partial f^*}{\partial z}$. 
As $||\mu||_\infty<1$, there exists
a quasiconformal homeomorphism of the sphere $\psi$, such that
the complex dilatation of $\psi$ is $\mu$, in particular, it is 
holomoprhic near infinity, and normalized
by $\psi(z)=\tilde a z + \tilde b + O(1/z)$ with any prescribed 
$\tilde a\not=0, \tilde b$.
To be more precise, if $\psi^\mu(z)$ is the normalized quasiconformal
map as in the beginning of Sect.~\ref{th4}, 
then $\psi^\mu(z)=\rho z+k+O(1/z)$
at $\infty$, and we define $\psi=a \psi^\mu + b$ with 
$a=\sigma_0/(\sigma \rho)$ and $b=(m_0-m)/\sigma - \sigma_0 k/(\sigma \rho)$.
For every $z$, $\psi^\mu(z)$ is holomorphic in $\bar v$, and,
it follows, that $\rho$, $k$ depend
holomorphically on $\bar v$, too. Therefore, $\psi(z)$ is also holomorphic in
$\bar v$.
Finally, define $f=f^*\circ \psi^{-1}$.
Then $f$ is rational. Moreover, $a,b$ are chosen so that 
$f(z)=\sigma z + m +O(1/z)$. It is easy to see that $f(z)$ 
is a continuous function of $\bar v$, and $f=f_0$ for $\bar v=\bar v(f_0)$.
Furthermore, 
$f(z)$ depends holomorphically on each variable
$\sigma, m, v_1,...$, for every $z$.
One can check this as in~\cite{er}: we differentiate the
identity $f\circ \psi=f^*$ by $\bar\partial\sigma,
\bar\partial m, \bar\partial v_1,...$ 
and take into account that $\psi$, $f^*$ are holomorphic in $\bar v$. 
Thus $f\in \Lambda_{d, \bar p'}$, $\bar v(f)=\bar v$,
and $f(z)$ depends on $\bar v$ holomorphically for every $z$.
Since $f_0(z)=\sigma_0 z+m_0+P_0(z)/Q_0(z)$, where 
the polynomials $P_0$ and $Q_0$ have no common roots,
this implies that the 
vector $\bar f$ of the coefficients of $f$ depends holomorphically
on $\bar v$ as well.
It defines a local holomorphic inverse $\delta^{-1}$.
By the above, we are done.
\section{Proof of Theorem~\ref{rteich}}
\subsection{Beltrami coefficients.}
As it has been mentioned already, we derive the theorem with help
of quasiconformal deformations. Let $f\in S_d$, and satisfy
the conditions (1)-(3) of Sect.~\ref{rthm}.
Let $\nu(z, t)$ be an
analytic family of invariant Beltrami
coefficients in the Riemann sphere, such that $\nu(z,0)=0$.
As $\nu(z,t)$ is differentiable at $t=0$, $\nu(z,t)=t\mu(z)+t\epsilon(z,t)$,
where $||\epsilon(z,t)||_\infty\to 0$ as $t\to 0$.
We have seen $\mu$ is invariant by $f$, too.
Let $h_t$ be an analytic family of quasiconformal
homeomorphisms in the plane that fix $\infty$, so that $h_t$ has
the complex dilatation $\nu(z,t)$, and $h_0=id$. 
Then $f_t=h_t\circ f\circ h_t^{-1}\in S_d$ is an analytic in $t$
family, and $O_t=h_t(O)$ the
corresponding attracting periodic orbit of $f_t$. Let $\rho(t)$ denote
its multiplier, and $v_j(t)=h_t(v_j)$ the 
critical values of $f_t$. Define also $\sigma(t)$, $m(t)$ by the expansion
$f_t(z)=\sigma(t)z+m(t)z+O(1/z)$ as $z\to \infty$.
\subsection{Speed of the multiplier.}
\begin{lem}\label{rlinspeed}
\begin{equation}\label{rlinspeedf}
\frac{\rho'(0)}{\rho}=-\frac{1}{\pi}\lim_{C\to \{b\}}\int_{C}
\frac{\mu(z)}{(z-b)^2}d\sigma_z,
\end{equation}
where $C$ is a fundamental region near $b\in O$. 
\begin{equation}\label{rlinspeedfinf}
\frac{\sigma'(0)}{\sigma}=\frac{1}{\pi}\lim_{C_\infty\to
\{\infty\}}\int_{C_\infty} \frac{\mu(z)}{z^2}d\sigma_z,
\end{equation}
where $C_\infty$ is a fundamental region near $\infty$.
\end{lem}
The first equality is the content of Lemma~\ref{linspeed}.
The~(\ref{rlinspeedfinf}) can be obtained from~(\ref{rlinspeedf})
by the change of variable $z\mapsto 1/z$. 
\subsection{Adjoint identity.}\label{adjiden}
We want to integrate the identity
\begin{equation}\label{ragain}
A(z)-(TA)(z)= \sum_{j=1}^{2d-2}\frac{L_j}{z-v_j}.
\end{equation}
against the $f$-invariant Beltrami form $\mu$. One cannot do this
directly, because $A$ is not integrable at the points of the
periodic orbit $O$ as well as at $\infty$ (if $\mu\not=0$ near
$\infty$). To deal with this, for every small $r>0$ and
big $R$, consider the domain $V_{r,R}$ to be the plane $\Ci$ with
the following sets deleted: $f(B^*(R))$ and $B(b_1,r)$ union with
$f_{b_{n-k+1}}^{-k}(B(b_1,r))$, for $k=1,...,n-1$, where
$f_{b_{n-k+1}}^{-k}$ is a local branch of $f^{-k}$ taking $b_1\in
O$ to $b_{n-k+1}$. In other words,
$$V_{r,R}={\Ci}\setminus \{f(B^*(R))\cup B(b_1,r)\cup_{k=1}^{n-1}
f_{b_{n-k+1}}^{-k}(B(b_1,r))\}.$$ Then $A$ is integrable in
$V_{r,R}$, and, therefore,
$$\int_{V_{r,R}}TA(z)\mu(z)d\sigma_z=\int_{f^{-1}(V_{r,R})}
A(z)\mu(z).$$ Now, $f^{-1}(V_{r,R})=V_{r,R}\setminus(C_r\cup
C^*_R\cup \Delta_r\cup \Delta^*_R)$, where
$C_r=f_{b_1}^{-n}(B(b_1,r))\setminus B(b_1,r)$ is a fundamental
region near $b_1$, and $C^*_R=B^*(R)\setminus f(B^*(R))$ is a
fundamental region near infinity (defined by the local branches
$f_{b_1}^{-n}$ that fixes $b_1$ and $f_\infty^{-1}$ that fixes
$\infty$ resp.), and, in turn, $\Delta_r$ and $\Delta^*_R$ are
open set which are away from $O$ and $\infty$, and which shrink to
a finitely many points as $r\to 0$ and $R\to \infty$ resp.
Therefore,
\begin{equation}\label{rinta}
\int_{V_{r,R}}(A(z)-TA(z))\mu(z)d\sigma_z=
\int_{C_r}A(z)\mu(z)d\sigma_z+\int_{C^*_R}A(z)\mu(z)d\sigma_z+o_r(1)+o^{\infty}_R(1).
\end{equation}
Here and below little-o notation mean that
$o_r(1)\to 0$ as $r\to 0$ and $o^{\infty}_R(1)\to 0$ as $R\to
\infty$. It is easy to see that
$$\int_{C_r}A(z)\mu(z)d\sigma_z=\int_{C_r}\frac{\mu(z)}{(z-b_1)^2}d\sigma_z+o_r(1)$$
and
$$\int_{C^*_R}A(z)\mu(z)d\sigma_z=\Gamma_1\int_{C^*_R}\frac{\mu(z)}{z}+
\Gamma_2\int_{C^*_R}\frac{\mu(z)}{z^2}+o^{\infty}_R(1),$$ 
where $\Gamma_1$,
$\Gamma_2$ are defined by the expansion
$A(z)=\Gamma_1/z+\Gamma_2/z^2+O(1/z^3)$ at infinity. Thus,
\begin{equation}\label{rintaex}
\int_{V_{r,R}}(A(z)-TA(z))\mu(z)d\sigma_z=
\int_{C_r}\frac{\mu(z)}{(z-b_1)^2}d\sigma_z+
\Gamma_1\int_{C^*_R}\frac{\mu(z)}{z}d\sigma_z+
\Gamma_2\int_{C^*_R}\frac{\mu(z)}{z^2}d\sigma_z+ o_r(1)+o^{\infty}_R(1).
\end{equation}
The identity~(\ref{ragain}) then gives us:
\begin{equation}\label{rmua}
\int_{C_r}\frac{\mu(z)}{(z-b_1)^2}d\sigma_z+
\Gamma_1\int_{C^*_R}\frac{\mu(z)}{z}d\sigma_z+
\Gamma_2\int_{C^*_R}\frac{\mu(z)}{z^2}d\sigma_z+ o_r(1)+o^{\infty}_R(1)=
\sum_{j=1}^{2d-2}L_j\int_{V_{r,R}}\frac{\mu(z)}{z-v_j}d\sigma_z
\end{equation}
Lemma~\ref{rlinspeed} allows us to pass to the limit as $r\to 0$:
\begin{equation}\label{radjidentity}
-\pi\frac{\rho'(0)}{\rho}+
\Gamma_1\int_{C^*_R}\frac{\mu(z)}{z}d\sigma_z+
\Gamma_2\int_{C^*_R}\frac{\mu(z)}{z^2}d\sigma_z+o^{\infty}_R(1)=
\sum_{j=1}^{2d-2}L_j\int_{V_{R}}\frac{\mu(z)}{z-v_j}d\sigma_z,
\end{equation}
where
$$V_R={\Ci}\setminus f(B^*(R)).$$
By the same Lemma~\ref{rlinspeed}, in the equation~(\ref{radjidentity}) one
can write the asymptotics as $R\to \infty$. We get:
\begin{equation}\label{radjidentityrat}
-\pi\frac{\rho'(0)}{\rho}+\Gamma_2\pi\frac{\sigma'(0)}{\sigma}+
\Gamma_1\int_{C^*_R}\frac{\mu(z)}{z}d\sigma_z+o^{\infty}_R(1)=
\sum_{j=1}^{2d-2}L_j\int_{V_{R}}\frac{\mu(z)}{z-v_j}d\sigma_z.
\end{equation}
\paragraph{Speed of critical values.}
Now we want to express the integral of $\mu(z)/(z-v_j)$ via
$v_j'(0)$. 
The difference with the polynomial case is that
$\mu$ does not vanish at infinity anymore.
Let $\psi_t$ be the quasiconformal homeomorphism of
the plane with the complex dilatation $\nu(z,t)$, that fixes $0, 1$
and $\infty$. As $h_t$ has the same complex dilatation and
fixes $\infty$ too, we have: $h_t=a(t)\psi_t+b(t)$, where $a$,
$b$ are analytic in $t$, and $a(0)=1$,
$b(0)=0$. Using (\ref{h})-(\ref{h'}), we can write
\begin{equation}\label{rv'}
v_j'(0)=a'(0)v_j+b'(0)-\frac{1}{\pi}\int_{V_R}\frac{\mu(z)}{z-v_j}d\sigma_z
-\frac{1}{\pi}\int_{V_R}\mu(z)(\frac{v_j-1}{z}-\frac{v_j}{z-1})d\sigma_z+
o^{\infty}_R(1).
\end{equation}
From this and~(\ref{radjidentity}), we obtain, then,
\begin{equation}\label{radjidenv}
\frac{\rho'(0)}{\rho}=\Gamma_2\frac{\sigma'(0)}{\sigma}+\sum_{j=1}^{2d-2}L_j
v_j'(0) +\Delta,
\end{equation}
where
$$\Delta=-b'(0)\sum_{j=1}^{2d-2}L_j-a'(0)\sum_{j=1}^{2d-2}v_j
L_j+$$
$$\lim_{R\to
\infty}\{\frac{\Gamma_1}{\pi}\int_{C^*_R}\frac{\mu(z)}{z}d\sigma_z-
\frac{1}{\pi}\int_{V_R}\frac{\mu(z)}{z}d\sigma_z
\sum_{j=1}^{2d-2}L_j-
\frac{1}{\pi}\int_{V_R}\frac{\mu(z)}{z(z-1)}d\sigma_z
\sum_{j=1}^{2d-2}v_jL_j\} .$$
Let us find connections between $\Gamma_1$, $\Gamma_2$, and $L_j$,
$v_j$:
\begin{lem}\label{rconn}
\begin{equation}\label{rconn1}
\frac{\sigma-1}{\sigma}\Gamma_1=\sum_{j=1}^{2d-2} L_j, \ \
\frac{m}{\sigma}\Gamma_1=-\sum_{j=1}^{2d-2} v_j L_j.
\end{equation}
\end{lem}
\begin{proof}
By definition, $A(z)=\Gamma_1/z+\Gamma_2/z^2+O(1/z^3)$. If a point
$e$ of the plane is such that $f(e)=\infty$, and $w$ is close to
$e$, so that $z=f(w)$ is close to $\infty$, then it is easy to
check that $f'(w)\sim C z^{2}$ as $z\to \infty$, where $C\not=0$.
We see that the asymptotics of $TA(z)$ at $\infty$ up
to $1/z^2$ is defined by the preimage $w_\infty$ of $z$, which is
close to $\infty$. In turn, $w_\infty=(z-m)/\sigma+O(1/z)$. It
gives us:
\begin{equation}\label{rta}
TA(z)=
\frac{\Gamma_1}{\sigma}\frac{1}{z}+(\Gamma_2+\frac{m\Gamma_1}{\sigma})\frac{1}{z^2}+O(\frac{1}{z^3}).
\end{equation}
Therefore,
\begin{equation}\label{ra-ta}
A(z)-TA(z)=\frac{\sigma-1}{\sigma}\Gamma_1\frac{1}{z}-\frac{m\Gamma_1}{\sigma}\frac{1}{z^2}+O(\frac{1}{z^3}).
\end{equation}
Comparing the latter asymptotics with the asymptotics at $\infty$
of $\sum_{j}\frac{L_j}{z-v_j}$, we get the statement.
\end{proof}
Note that~(\ref{rta})-(\ref{ra-ta}) will be used also later on in the
proof of Theorem~\ref{rattr}.

Let us continue.
In view of the latter connections, we can write that
\begin{equation}
\Delta=\frac{\Gamma_1}{\sigma}\Delta_0,
\end{equation}
and
\begin{equation}\label{rdeltaup}
\Delta_0=m a'(0)-(\sigma-1)b'(0) +\frac{1}{\pi}\lim_{R\to
\infty}\{\sigma\int_{C^*_R}\frac{\mu(z)}{z}d\sigma_z-(\sigma-1)\int_{V_R}\frac{\mu(z)}{z}d\sigma_z
+m\int_{V_R}\frac{\mu(z)}{z(z-1)}d\sigma_z\},
\end{equation}
where
$$V_R={\Ci}\setminus f(B^*(R)), \ \ \ C^*_R=B^*(R)\setminus
f(B^*(R)).$$ Our aim is to show that
$$\Delta_0=m'(0).$$
\paragraph{Evaluation of $m'(0)$.}
Here we solve the following general problem:
calculate $m'(0)$, where $f_t=h_t\circ f\circ h_t^{-1}$ is the quasiconformal
deformation of $f$, such that $f_t(z)=\sigma(t)z+m(t)+O(1/z)$ at $\infty$.
(Note that $\sigma'(0)$ has been calculated in Lemma~\ref{rlinspeed}.)
We get
from $f_t\circ h_t=h_t\circ f$, that $f_t(z)=f(z)+t V(z) + O(t^2)$
with
\begin{equation}\label{rm1}
V(z)=a'(0)(f(z)-z
f'(z))+b'(0)(1-f'(z))+\kappa(f(z))-f'(z)\kappa(z),
\end{equation}
and
\begin{equation}\label{rm2}
\kappa(z)=\frac{\partial \psi_t}{\partial
t}|_{t=0}(z)=-\frac{1}{\pi}\int_{\Ci}\mu(w)\frac{z(z-1)}{w(w-1)(w-z)}d\sigma_w.
\end{equation}
Using the asymptotics of $f$ at $\infty$, we proceed:
\begin{equation}\label{rm3}
V(z)=m a'(0)+(1-\sigma) b'(0) + \kappa(f(z))-f'(z)\kappa(z) +
O(\frac{1}{z}).
\end{equation}
Note that $V$ is a rational function of $z$. In particular, it is
meromorphic at $\infty$. Therefore,
\begin{equation}\label{rm4}
m'(0)=\frac{1}{2\pi i}\int_{|z|=R} \frac{V(z)}{z}dz= m
a'(0)+(1-\sigma) b'(0) + \frac{1}{2\pi i}\int_{|z|=R}
\frac{\kappa(f(z))-f'(z)\kappa(z)}{z}dz,
\end{equation}
for every $R$ large enough.
Now we need to calculate
\begin{equation}\label{rm5}
J=\frac{1}{2\pi i}\int_{|z|=R}
\frac{\kappa(f(z))-f'(z)\kappa(z)}{z}dz.
\end{equation}
By~(\ref{rm2}) and Fubini's theorem:
\begin{equation}\label{rm6}
J=-\frac{1}{\pi}\int_{\Ci}\frac{\mu(w)}{w(w-1)}d\sigma_w
\frac{1}{2\pi
i}\int_{|z|=R}\frac{1}{z}[\frac{f(z)(f(z)-1)}{w-f(z)}-\frac{f'(z)z(z-1)}{w-z}]dz.
\end{equation}
Denote the internal integral
\begin{equation}\label{rm7}
I(w,R)=\frac{1}{2\pi i}\int_{|z|=R}F(z,w)dz,
\end{equation}
where
\begin{equation}\label{rm8}
F(z,w)=\frac{1}{z}[\frac{f(z)(f(z)-1)}{w-f(z)}-\frac{f'(z)z(z-1)}{w-z}].
\end{equation}
For a large (but fixed!) $R$, we calculate $I(w,R)$ for different
$w$ using the Residue Theorem. First, it is easy to check that, as
$z\to \infty$,
$$F(z,w)=\frac{1}{z}\frac{w[f(z)(f(z)-1)-f'(z)z(z-1)]+[f'(z) f(z) z
(z-1)-z f(z) (f(z)-1)]}{(w-f(z))(w-z)}=$$
$$=\frac{(\sigma-1)(w-1)-m}{z} +O(\frac{1}{z^2}).$$
Therefore,
\begin{equation}\label{rm10}
I(w,R)=[(\sigma-1)(w-1)-m]-\sum_{|z|>R} Res F(z,w).
\end{equation}
The result depends on the position of $w$.

(i) If $|w|<R$, then $F(z,w)$ has no singular points for $|z|\ge
R$. Therefore,
\begin{equation}\label{rm11}
I(w,R)=(\sigma-1)(w-1)-m.
\end{equation}

(ii) If $|w|>R$ and $w\in {\Ci}\setminus f(B^*(R))$, then $F(z,w)$
has a single singular point in $|z|>R$ at $z=w$. Therefore,
\begin{equation}\label{rm12}
I(w,R)=[(\sigma-1)(w-1)-m] - f'(w)(w-1)=-(w-1)-m+O(\frac{1}{w}).
\end{equation}

(iii) If $w\in f(B^*(R))$, then $F(z,w)$ has two singular points
in $|z|>R$: at $z=w$ and at a unique $z=z_w$ in this domain, such
that $f(z_w)=w$. Therefore,
\begin{equation}\label{rm13}
I(w,R)=[(\sigma-1)(w-1)-m]-f'(w)(w-1)+\frac{w(w-1)}{z_w f'(z_w)}.
\end{equation}
We have: $z_w=(w-m)/\sigma +O(1/w)$ so, hence, after some 
straightforward manipulations,
\begin{equation}\label{rm14}
I(w,R)=[(\sigma-1)(w-1)-m]-f'(w)(w-1)+\frac{w(w-1)}{z_w
f'(z_w)}=O(\frac{1}{w}).
\end{equation}
With help of (i)-(iii), we calculate
\begin{equation}\label{rm15}
J=-\frac{1}{\pi}\int_{\Ci}\frac{\mu(w)}{w(w-1)}I(w,R)d\sigma_w=
J_1+J_2+J_3
\end{equation}
as follows.
$$J_1=-\frac{1}{\pi}\int_{|w|<R}\frac{\mu(w)}{w(w-1)}[\sigma-1)(w-1)-m]d\sigma_w=$$
$$-\frac{1}{\pi}(\sigma-1)\int_{|w|<R}\frac{\mu(w)}{w}d\sigma_w+
\frac{1}{\pi}m\int_{|w|<R}\frac{\mu(w)}{w(w-1)}d\sigma_w,$$
and
$$J_2=-\frac{1}{\pi}\int_{B^*(R)\setminus
f(B^*(R))}\frac{\mu(w)}{w(w-1)}[-(w-1)-m+O(\frac{1}{w})]d\sigma_w=$$
$$\frac{1}{\pi}\int_{B^*(R)\setminus
f(B^*(R))}\frac{\mu(w)}{w}d\sigma_w+\frac{m}{\pi}\int_{B^*(R)\setminus
f(B^*(R))}\frac{\mu(w)}{w(w-1)}d\sigma_w+\int_{B^*(R)\setminus
f(B^*(R))}O(\frac{1}{w^3})d\sigma_w,$$

and, at last,

$$J_3=-\frac{1}{\pi}\int_{f(B^*(R))}\frac{\mu(w)}{w(w-1)}O(\frac{1}{w})d\sigma_w
=\int_{f(B^*(R))}O(\frac{1}{w^3})d\sigma_w.$$

Since $J$ is independent on $R$, we can write then:
$$m'(0)=m a'(0)+(1-\sigma) b'(0) - \frac{1}{\pi}\lim_{R\to \infty}
\{(\sigma-1)\int_{|w|<R}\frac{\mu(w)}{w}d\sigma_w-\int_{B^*(R)\setminus
f(B^*(R)}\frac{\mu(w)}{w}d\sigma_w-$$
$$m\int_{|w|<R}\frac{\mu(w)}{w(w-1)}d\sigma_w-m\int_{B^*(R)\setminus
f(B^*(R))}\frac{\mu(w)}{w(w-1)}d\sigma_w\}.$$
In other words, we have proved:
\begin{lem}\label{m'}
$$m'(0)=m a'(0)+(1-\sigma) b'(0) + $$
$$\frac{1}{\pi}\lim_{R\to \infty} \{-(\sigma-1)\int_{{\Ci}\setminus
f(B^*(R))}\frac{\mu(w)}{w}d\sigma_w+m\int_{{\Ci}\setminus
f(B^*(R))}\frac{\mu(w)}{w(w-1)}d\sigma_w+
\sigma\int_{B^*(R)\setminus
f(B^*(R))}\frac{\mu(w)}{w}d\sigma_w\}.$$
\end{lem}

\

If we compare the latter expression for $m'(0)$ to the
expression~(\ref{rdeltaup}) for $\Delta_0$, we see that 
$\Delta_0=m'(0)$. This finishes the proof of Theorem~\ref{rteich}, and,
therefore,
Theorem~\ref{rmain}, for $f\in S_d$.
\section{More generality: multiple critical points yet finite
critical values}\label{multiple} 

Let $f\in \Lambda_{d, \bar p}$, and assume that
all critical values of $f$ are finite: if $c_j$ are all different
critical points of $f$, then $v_j=f(c_j)\not=\infty$, $1\le j\le
p$. As usual, $m_j$ is the multiplicity of $c_j$. By
Theorem~\ref{rformula}, we have:
\begin{equation}\label{rruellehyp2}
B(z)-(TB)(z)= \sum_{j=1}^{p}\frac{\tilde L_j}{z-v_j},
\end{equation}
where
\begin{equation}\label{lj2}
\tilde L_j=
-\frac{1}{(m_j-1)!}\frac{d^{m_j-1}}{dw^{m_j-1}}|_{w=c_j}(\frac{B(w)}{Q_j(w)}),
\end{equation}
where, in turn, $Q_j$ is a local analytic function near $c_j$
defined by $f'(z)=(z-c_j)^{m_j}Q_j(z)$, so that $Q_j(c_j)\not=0$.
What we need to check is that
\begin{equation}\label{ljr}
\tilde L_j=\frac{\partial \rho}{\partial v_j}.
\end{equation}
Without loss of generality, $j=1$. The idea is as follows. Using
the coordinates $\bar v(f)$, we consider a small analytic path
$f_t$ through $f$ in the space $\Lambda_{d, \bar p}$, such that
only $v_1$ changes along this path. Then we perturb each $f_t$ in
such a way, that all critical points of the perturbed map are
simple, and apply Theorem~\ref{rmain} in the proven case of simple
critical points. Then we get an integral formula for the variation of
$\rho$ along the path, which will imply~(\ref{ljr}).

Let us do the required analytic work. By Proposition~\ref{rlocal},
there exists a family $\{f_t\}$ of maps from $\Lambda_{d, \bar p}$
of a complex parameter $t$, $|t|<\delta$, where $\delta>0$ is
small, such that $f_0=f$ and, for every $t$, $\bar
v(f_t)=\{\sigma, m, v_1+t\delta, v_2,...,v_p\}$. Denote by
$c_j(t)$ a critical point of $f_t$, the continuation of the
critical point $c_j$. If $\delta$ is small enough, the periodic
orbit $O$ of $f$ extends holomorphically to a periodic orbit $O_t$
of $f_t$. We perturb $f_t$ as follows: given $\epsilon$ with small
enough modulus, define $f_{t, \epsilon}(z)=f_t(z)+\epsilon z$. It
is easy to see, that all critical points of $f_{t, \epsilon}$ are
simple, that is, $f_{t, \epsilon}\in S_d$. 
To be more precise, for every $|\epsilon|\not=0$ small
enough, to every critical point $c_j(t)$ of $f_t$ there
corresponds $m_j$ simple critical points $c_{j,k}(t, \epsilon)$,
$k=1,...,m_j$, of $f_{t, \epsilon}$, so that $\lim_{\epsilon\to 0}
c_{j, k}(t, \epsilon)=c_j(t)$. Denote $v_{j, k}(t, \epsilon)=f_{t,
\epsilon}(c_{j, k}(t, \epsilon))$. Furthermore, if $\delta$,
$|\epsilon|$ are small, the periodic points of $O$ extends to
holomorphic functions of $t, \epsilon$ and form a periodic orbit
$O(t, \epsilon)$ of $f_{t, \epsilon}$. Its multiplier $\hat
\rho(t, \epsilon)$ is a holomorphic function in $t, \epsilon$,
too. In particular, $\hat \rho(t, 0)=\hat \rho(t)$, the multiplier
of $O_t$ for $f_t$. We have:
$$\hat \rho(t)-\hat \rho(0)=
\lim_{\epsilon\to 0}\{\hat \rho(t, \epsilon)-\hat \rho(0, \epsilon)\}.$$ 
We fix $\epsilon\not=0$ and calculate $\hat \rho(t,
\epsilon)-\hat \rho(0, \epsilon)$ as follows. Since all critical
points of $f_{t, \epsilon}$ are simple, then 
$\hat \rho(t,\epsilon)$ is a holomorphic function $\rho(\bar v(f_{t,
\epsilon}))$ of $\bar v(f_{t, \epsilon})=\{\sigma(t, \epsilon),
m(t, \epsilon), \{v_{j, k}(t, \epsilon)\}_{j=1,...,p;
k=1,...,m_j}\}$. Now, $f_{t,
\epsilon}(z)=(\sigma+\epsilon)z+m+O(1/z)$, hence, $\sigma(t,
\epsilon)=\sigma+\epsilon, m(t, \epsilon)=m$.

We will denote by $(z)_t=\partial z/\partial t$. 
In particular, $(\sigma)_t(t, \epsilon)=(m)_t(t,
\epsilon)=0$. We have:
$$\hat \rho(t_0, \epsilon)-\hat \rho(0,
\epsilon)=\int_0^{t_0}(\hat \rho)_t(t, \epsilon)dt=$$
$$\int_0^{t_0} (\frac{\partial\rho}{\partial \sigma}(\sigma)_t(t,
\epsilon)+ \frac{\partial\rho}{\partial m} (m)_t(t, \epsilon)+
\sum_{j=1}^p\sum_{k=1}^{m_j}\frac{\partial\rho}{\partial v_{j, k}}
(v_{j, k})_t(t, \epsilon))dt=\int_0^{t_0}
(\sum_{j=1}^p\sum_{k=1}^{m_j}\frac{\partial\rho}{\partial v_{j,
k}} (v_{j, k})_t(t, \epsilon))dt.$$ 
By the proven formulae for 
$f_{t, \epsilon}\in S_d$,
$$\frac{\partial \rho}{\partial v_{j,k}}=-
\frac{B_{O(t,\epsilon)}(c_{j, k}(t, \epsilon))}{f_{t,
\epsilon}"(c_{j, k}(t, \epsilon))}.$$
Also,
$$(v_{j, k})_t(t, \epsilon)=\frac{d}{dt}f_{t, \epsilon}(c_{j, k}(t,
\epsilon))=(f_{t, \epsilon})_t(c_{j, k}(t, \epsilon))+
f_{t,\epsilon}'(c_{j, k}(t, \epsilon))(c_{j, k})_t(t, \epsilon)= 
(f_{t,\epsilon})_t(c_{j, k}(t, \epsilon)).$$ Taking this into account, we
can find $r>0$ fixed and write:
$$\hat \rho(t_0, \epsilon)-\hat \rho(0,
\epsilon)=-\int_0^{t_0} (\sum_{j=1}^p\sum_{k=1}^{m_j}\frac{B_{O(t,
\epsilon)}(c_{j, k}(t, \epsilon))}{f_{t, \epsilon}"(c_{j, k}(t,
\epsilon))} (f_{t, \epsilon})_t(c_{j, k}(t, \epsilon)))dt=$$
$$-\int_{0}^{t_0}(\sum_{j=1}^p\frac{1}{2\pi
i}\int_{|w-c_j(t)|=r}\frac{B_{O(t, \epsilon)}(w)}{f_{t,
\epsilon}'(w)}(f_{t, \epsilon})_t(w)dw)dt.$$ Passing to a limit as
$\epsilon\to 0$, we obtain:
$$\hat\rho(t_0)-\hat\rho(0)=-\int_{0}^{t_0}(\sum_{j=1}^p\frac{1}{2\pi
i}\int_{|w-c_j(t)|=r}\frac{B_{O_t}(w)}{f_{t}'(w)}(f_{t})_t(w)dw)dt.$$
It follows from $f_t(w)=v_j(t)+O(w-c_j(t))^{m_j+1}$, that
$$(f_t)_t(w)=(v_j)_t(t)+O(w-c_j(t))^{m_j}.$$
On the other hand, by the choice of $f_t$, $(v_1)_t(t)=1$ while
$(v_j)_t(t)=0$ for $2\le j\le p$. Therefore,
$$\hat \rho(t_0)-\hat \rho(0)=\int_{0}^{t_0} L_1(t)dt,$$
where
$$L_1(t)=-\frac{1}{2\pi
i}\int_{|w-c_1(t)|=r}\frac{B_{O_t}(w)}{f_t'(w)}dw.$$ In the local
coordinates of $\Lambda_{d, \bar p}$ near $f_0=f$,
$$\rho(\sigma, m, v_1+t_0, v_2,...,v_p)-\rho(\sigma, m,
v_1,...,v_p)=\int_{0}^{t_0} L_1(t)dt.$$ It yields immediately,
that
$$\frac{\partial \rho}{\partial v_1}=-\frac{1}{2\pi
i}\int_{|w-c_1|=r}\frac{B(w)}{f'(w)}dw=L_1(0).$$
\section{Finishing the proof of Theorem~\ref{rmain}:
the case of infinite critical values}\label{infcrv} 
Let $f_0\in
\Lambda_{d, \bar p'}$ have some critical values equal to $\infty$,
and $O$ be a periodic orbit of $f_0$ with the multiplier different
from $1$. Without loss of generality one can assume that $f_0$ has
different critical points $c_1^0,...,c_{q'}^0$, so that the
critical values $v_j^0=f_0(c_j^0)$ are finite for $1\le j\le
p<p'$, and $v_j^0=\infty$ for $p<j\le p'$. As usual, $m_j$ denotes
the multiplicity of $c^0_j$. By Proposition~\ref{rlocal}, $\bar
v(f)=(\sigma, m, v_1,...,v_p,1/v_{p+1},...,1/v_{p'})$ are local
coordinates in $\Lambda_{d, \bar p'}$ in a neighborhood of $\bar
v(f_0)=(\sigma_0, m_0, v_1^0,...,v_p^0,0,...,0)\in {\bf
C}^{p'+2}$. Therefore, $f$ and $\rho$ are holomorphic in $\bar
v(f)$ in a neighborhood of $\bar v(f_0)$. In particular, for every
$1\le j\le p'$, $\partial \rho/\partial v_j$ is continuous in
$\bar v(f)$ at $\bar v(f_0)$. Consider now a sequence $f_n\in
\Lambda_{d, \bar p'}$, such that all critical values of each $f_n$
are finite, and $\bar v(f_n)\to \bar v(f_0)$. Let $O_n$ be the
periodic orbit of $f_n$, such that $O_n\to O$ as $n\to \infty$,
$\rho_n$ the multiplier of $O_n$, and let $B_n$ and $T_n$ be the associated
rational function for $O_n$ and the transfer operator for $f_n$.
Clearly, $B_n\to B$, $T_n\to T$, the corresponding objects for
$f_0$. To prove Theorem~\ref{rmain} for $f_0$, it remains to show two
facts:

(a) $\partial \rho_n/\partial v_j\to 0$, $n\to \infty$, for
$p<j\le p'$,

(b) 
$$\frac{\partial \rho_n}{\partial (1/v_j)}=-v_j(f_n)^2\frac{\partial
\rho_n}{\partial v_j}\to \frac{1}{2\pi
i}\int_{|w-c^0_j|=r}\frac{B(w)}{(1/f_0)'(w)}dw$$
$$=\frac{1}{(m_j-1)!}\frac{d^{m_j}}{d
w^{m_j}}|_{w=c^0_j}(\frac{B}{Q_{j}}),$$
$n\to \infty$, for $p<j\le p'$, where
$(1/f_0)'(w)=(w-c^0_j)^{m_j}Q_j(w)$.

Since $v_j(f_n)\to \infty$, $p<j\le p'$, then (a) follows from
(b). To prove (b), let us fix $p<j\le p'$. Then fix $n$, and
denote $c=c_j(f_n)$, $m=m_j$, and $Q$, such that $f_n'(w)=(w-c)^m
Q(w)$. As $w\to c$, we can write
$$\frac{(f_n(w))^2 B_n(w)}{(f_n'(w))^2}=
\frac{(f_n(c)+O(w-c)^{m+1})^2 B_n(w)}{(w-c)^m
(Q(w))^2}=\frac{(f_n(c))^2}{(w-c)^m (Q(w))^2}+O(w-c).$$ Therefore,
for $r_n>0$ small enough and every $p<j\le p'$,
$$(v_j(f_n))^2\frac{\partial \rho_n}{\partial v_j}=-\frac{1}{2\pi
i}\int_{|w-c_j(f_n)|=r_n}\frac{(f_n(w))^2 B_n(w)}{f_n'(w)}dw.$$
For all $n$ large, one can increase $r_n$ to a size, which is
independent on $n$. Indeed, $f_n\to f_0$ in $\Lambda_{d, \bar
p'}$. Hence, for a point $w_0$ near $c_j(f_n)$, if
$f_n(w_0)=\infty$, then $f_n\sim C_n/(w-w_0)$, $w\to w_0$, and
$(f_n(w))^2/f_n'(w)$ has no singularity at $w_0$. This shows that
one can fix $r_n=r>0$. Then, for every $p<j\le p'$ and fixed $r$,
we can pass to a limit as $n\to \infty$:
$$-v_j(f_n)^2\frac{\partial \rho_n}{\partial v_j}=\frac{1}{2\pi
i}\int_{|w-c_j^0|=r}\frac{(f_n(w))^2 B_n(w)}{f_n'(w)}dw \to
\frac{1}{2\pi i}\int_{|w-c^0_j|=r}\frac{(f_0(w))^2
B(w)}{f_0'(w)}dw.$$
The proof of Theorem~\ref{rmain} is completed.
\section{Multipliers as local coordinates}\label{rattrs}
\subsection{Contraction of the operator $T$}
We need the following Proposition~\ref{rcontr}.
A more general statement, with a careful
consideration of parabolic points, is contained in~\cite{Eps}.
\begin{prop}\label{rcontr}
Let $P$ be a non-empty union of some non-repelling periodic orbits
of $f\in \Lambda_{d, \bar p}$. Consider a non-zero
rational function $\psi$, such that:

(i) as $z\to \infty$, one of the following conditions hold:
either (a) $\psi(z)=O(1/z^3)$, or (b) $\psi(z)=O(1/z)$ and 
$f(z)=\sigma z+O(1/z)$, where $|\sigma|\ge 1$,
or (c) $\psi(z)=O(1/z^2)$ and $f(z)=z+m+O(1/z)$, 

(ii) if $b\in P$ is a point of a periodic orbit $O$, then $\psi$ has either
double pole at every point of $O$, or at most simple pole at every point of
$O$; moreover, $\psi$ has at most simple poles outside of the set $P$.

Then $\psi$ is not a fixed point of the operator $T$.
\end{prop} 
\begin{proof} (cf.~\cite{Eps},~\cite{dht}, 
see also~\cite{Leij0},~\cite{leij1})
Denote $\hat P=\cup_{j}O_j$, 
where $O_j$ are different
periodic orbits from the set $P$, such that $\psi$ has double pole at every
point of $O_j=\{b_k^j\}_{k=1}^{n_j}$. (If $\psi$ has no double
poles, then $\hat P=\emptyset$.) 
Denote by $\rho_j$ the multiplier of $O_j$.
Given $r>0$ small enough, we define a domain $V_r={\bf C}\setminus
\{W_\infty\cup \cup_j W_j\}$, where:

(i) if $\psi(z)=O(1/z^3)$ at infinity, then $W_\infty$ is empty,
and 
otherwise define
$W_\infty=B^*(1/r)$, a neighborhood of $\infty$.

(ii) for every $O_j$, the set $W_j$ is the disk $B(b_1^j,r)$
union with $f_{j}^{-k}(B(b^j_1,r))$, for
$1\le k\le n_j-1$, where $f_{j}^{-k}$ is a local branch
of $f^{-k}$ taking $b^j_1\in O_j$ to $b^j_{n_j-k+1}$.

Obviously, $V_r\to {\bf C}$ as $r\to 0$, and $\psi$ is
integrable in $V_r$, for $r>0$.
Let $j'$, $j"$, and $j'''$ denote the indexes corresponding to
attracting, neutral, and superattracting periodic orbits in $\hat P$
respectively. For every $j$, fix a positive number $a_j$, such that
$|f^{n_j}(z)-b_1^j-\rho_j (z-b_1^j)|<a_j|z-b_1^j|^2$,
for $|z-b_1^j|<r$, $r$ small enough. Also, for
$f(z)=\sigma z+m+O(1/z)$, fix $a>0$, such that $|f(z)-\sigma z-m|<a/|z|$, for
$|z|$ large. Now we have:
$$f^{-1}(V_r)\subset V_r\cup \hat W_\infty \cup \cup_{j"}\hat W_{j"}
\setminus \{\cup_{j'}
\hat W_{j'}\cup \cup_{j'''}\hat W_{j'''}\}.$$
Here:

(1) $\hat W_\infty$ is empty if either $W_\infty$ is empty
or $|\sigma|>1$,
and $\hat W_\infty\subset B^*(1/r)\setminus B^*(1/r+|m|+ar)$ otherwise
(here $m=0$ in the case (i)(b)).

(2) $\hat W_{j"}\subset B(b_1^{j"},r)\setminus B(b_1^{j"},
r-a_{j"}r^2)$, for every neutral periodic orbit $O_{j"}$.

(3) $B(b_1^{j'},|\rho_{j'}|^{-1}r-a_{j'}r^2)\setminus B(b_1^{j'},
r)\subset \hat W_{j'}$, for every attracting though not
superattracting periodic orbit $O_{j'}$.

(3') $B(b_1^{j'''},2r)\setminus B(b_1^{j'''}, r)\subset \hat
W_{j'''}$, if $O_{j'''}$ is superattracting.

It follows (see also the proof of Theorem~\ref{attr}) that:
\begin{equation}\label{rj''inf}
\lim_{r\to 0}\int_{\hat W_{\infty}}|\psi|d\sigma=0, \ \ 
\lim_{r\to 0}\int_{\hat W_{j"}}|\psi|d\sigma=0,
\end{equation}
\begin{equation}\label{rj''}
\liminf_{r\to 0}\int_{\hat W_{j'}}|\psi|d\sigma\ge
2\pi|A_{j'}|\log|\rho_{j'}|^{-1}, \ \  \liminf_{r\to 0}\int_{\hat W_{j'''}}
|\psi|d\sigma\ge 2\pi|A_{j'''}|\log 2,
\end{equation}
where $A_j\not=0$ is so that $\psi(z)\sim A_j/(z-b_1^j)^2$.
Therefore, under the conditions (i)-(ii),
\begin{equation}\label{rf-1}
\limsup_{r\to 0} \{\int_{f^{-1}(V_r)}|\psi(z)|d\sigma_z - \int_{V_r}
|\psi(z)|d\sigma_z\} \le 0,
\end{equation}
and, moreover, the inequality is strict, if $\psi$ has at least one double pole
at an attracting or superattracting point of $P$.
Assume now the contrary, i.e. $\psi=T\psi$.
If the inequality in~(\ref{rf-1}) is indeed strict, we get at once a contradiction
as in the proof of Theorem~\ref{attr}:
$0=\int_{V_r} |\psi|-\int_{V_r} |T\psi|\ge \int_{V_r} |\psi|-
\int_{f^{-1}(V_r)} |\psi|>0$, for some $r>0$.
This proves that $\psi$ cannot have double poles at attracting and
superattracting points of $P$.
Let us show further that
\begin{equation}\label{rtrieq}
|\sum_{w: f(w)=z}\frac{\psi(w)}{(f'(w))^2}|=\sum_{w:
f(w)=z}\frac{|\psi(w)|}{|f'(w)|^2},
\end{equation}
for every $z\in {\bf C}$, where both sides are finite.
Indeed, assume that, for some $z_0$, there is the strict
inequality in~(\ref{rtrieq}). Then there are a neighborhood $U$ of
$z_0$ and $\epsilon>0$, such that, for $z\in U$,
\begin{equation}\label{rtriineq}
|\sum_{w: f(w)=z}\frac{\psi(w)}{(f'(w))^2}|<(1-\epsilon)\sum_{w:
f(w)=z}\frac{|\psi(w)|}{|f'(w)|^2}.
\end{equation}
One writes, for $r>0$ small enough
(so that $U\subset V_r$):
$$\int_{V_r}|\psi(z)|d\sigma_z=\int_{V_r}|(T\psi)(z)|d\sigma_z=
\int_{V_r\setminus
U}|(T\psi)(z)|d\sigma_z+\int_{U}|(T\psi)(z)|d\sigma_z<$$
$$
\int_{f^{-1}(V_r\setminus
U)}|\psi(z)|d\sigma_z+(1-\epsilon)\int_{f^{-1}(U)}|\psi(z)|d\sigma_z=
\int_{f^{-1}(V_r)}|\psi(z)|d\sigma_z-\epsilon
\int_{f^{-1}(U)}|\psi(z)|d\sigma_z.$$
Taking into account~(\ref{rf-1}), we conclude that $\psi=0$ on
$U$, hence, everywhere. This contradiction shows that
~(\ref{rtrieq}) holds for every $z$ as above. 
In turn,~(\ref{rtrieq}) and $T\psi=\psi$ imply, that a meromorphic 
function $\psi\circ f (f')^2/\psi$ takes only positive values, hence, 
this function is 
the constant $d=\deg f$. Now, consider the identity 
$\psi(f^n(z)) ((f^n)'(z))^2=d^n \psi(z)$ 
near a point $b\in P$ of period $n$ and with multiplier $\rho$,
and plug in it the local expansion for $f^n$ and 
$\psi(z)\sim A(z-b)^l$, with $A\not=0$ and $l\ge -2$,
near $b$. Then 
we see that: if $\rho\not=0$, then $\rho^{l+2}=d^n$, i.e. $l>-2$
and $|\rho|>1$; if $\rho=0$, then $l=-2$. 
In either case, it is a contradiction. 
\end{proof}
\subsection{Proof of Theorem~\ref{rattr}}
Consider first the case ($H_\infty$).
Assume the contrary: the rank of the matrix
\begin{equation}\label{rma}
{\bf O}=(\frac{\partial \rho_j}{\partial \sigma}, \frac{\tilde{\partial}^V
\rho_j}{\partial V_1},..., \frac{\tilde{\partial}^V \rho_j}{\partial
V_{k-1}}, \frac{\tilde{\partial}^V \rho_j}{\partial V_{k+1}},
\frac{\tilde{\partial}^V \rho_j}{\partial V_q})_{1\le j\le r}
\end{equation}
is less than $r$. Without loss of generality,
one can assume that $k=q$. Then the vectors
$(\frac{\partial \rho_j}{\partial \sigma}, \frac{\tilde{\partial}^V
\rho_j}{\partial V_1},..., \frac{\tilde{\partial}^V \rho_j}{\partial
V_{q-1}})$, $1\le j\le r$, are linearly dependent.

Let $O_j=\{b^j_k\}_{k=1}^{n_j}$, the set of points of the periodic
orbit $O_j$ of period $n_j$, 
the function $\tilde B_j$ is said to be $B_{O_j}$ iff $\rho_j\not=1$
and $\hat B_{O_j}$ iff $\rho_j=1$.
Precisely like in the proof
of Theorem~\ref{attr}, Sect.~\ref{attrs}, we see that
each $\tilde B_j$ is not identically zero. 
The connections~(\ref{rruelleV}), (\ref{rruelleneutralV}) read:
$\tilde B_j(z)-(T\tilde B_j)(z)=\sum_{i=1}^q 
\frac{\tilde{\partial}^V \rho_j}{\partial
V_i}\frac{1}{z-V_i},$
for every $j=1,...,r$.
By the assumption, there exists a linear
combination $\psi=\sum_{j=1}^r \beta_j \tilde B_j$, such that the
following holds:
\begin{equation}\label{rl}
\psi(z)-(T\psi)(z)=\frac{L}{z-V_q},
\end{equation}
where $L=\sum_{j=1}^r \beta_j \frac{\tilde{\partial}^V \rho_j}{\partial
V_q}$. By~(\ref{rpartialrest}),
\begin{equation}\label{rpartialsigma}
\frac{\tilde{\partial} \rho_j}{\partial
\sigma}=\frac{\tilde\Gamma_2^j}{\sigma},
\end{equation}
where $\tilde\Gamma_2^j$ and also $\tilde\Gamma_1^j$ are defined
by the expansion
$\tilde B_j(z)=\frac{\tilde\Gamma_1^j}{z}+\frac{\tilde\Gamma_2^j}{z^2}+O(\frac{1}{z^3})$
at infinity. Therefore, if $M_1$, $M_2$ are defined by
$\psi(z)=\frac{M_1}{z}+\frac{M_2}{z^2}+O(\frac{1}{z^3})$
at infinity, then
\begin{equation}\label{rga}
M_2=\sum_{j=1}^r \beta_j \frac{\tilde{\partial} \rho_j}{\partial
\sigma}=0.
\end{equation}
Now, by~(\ref{rta}),
\begin{equation}\label{rtpsi}
T\psi(z)= \frac{M_1}{\sigma}\frac{1}{z}+(M_2+\frac{m
M_1}{\sigma})\frac{1}{z^2}+O(\frac{1}{z^3}),
\end{equation}
and
\begin{equation}\label{rpsi-tpsi}
\psi(z)-T\psi(z)=\frac{\sigma-1}{\sigma}M_1\frac{1}{z}-\frac{m
M_1}{\sigma}\frac{1}{z^2}+O(\frac{1}{z^3}).
\end{equation}
But $m=0$, which, together with~(\ref{rpsi-tpsi}) and~(\ref{rl})
imply
\begin{equation}\label{rlm}
L=M_1(1-\frac{1}{\sigma}), \ \ \ \ L V_q=0.
\end{equation}
Since $V_q\not=0$, then $L=0$. In other words, $\psi$ is a fixed point
of $T$. It satisfies the conditions of Proposition~\ref{rcontr}.
Indeed, since $\sigma\not=1$, ~(\ref{rlm}) then gives us that $M_1=0$, and 
then
$\psi(z)=O(1/z^3)$. Applying
Proposition~\ref{rcontr} (where the assumptions (i)(a)-(ii) hold), 
we get a contradiction.

Remaining cases are quite similar.

($H_\infty^{attr}$). One can assume $k=q$ and assuming that the rank 
of the martix ${\bf O^{attr}}$ is less than $r$ and using the notations
of the previous case, we obtain that $L=0$, i.e., $\psi$ is a non-trivial
fixed point of $T$. Applying Proposition~\ref{rcontr} with
the assumptions (i)(b)-(ii), get a contradiction. 

($NN_\infty$). 
One can assume $k=q$ and assume that the rank of the matrix
${\bf O^{neutral}}$ is less than $r$.
Using the notations from the first case,
there exists a linear
combination $\psi=\sum_{j=1}^r \beta_j \tilde B_j$, such that
\begin{equation}\label{rln}
\psi(z)-(T\psi)(z)=\frac{L}{z-V_q},
\end{equation}
where $L=\sum_{j=1}^r \beta_j \frac{\tilde{\partial}^V \rho_j}{\partial
V_q}$. 
If $M_1$, $M_2$ are defined by
$\psi(z)=\frac{M_1}{z}+\frac{M_2}{z^2}+O(\frac{1}{z^3})$
at infinity, then, by~(\ref{rta}),
\begin{equation}\label{rtpsin}
T\psi(z)=\frac{M_1}{z}+\frac{M_2+m M_1}{z^2}+O(\frac{1}{z^3}),
\end{equation}
and
\begin{equation}\label{rpsi-tpsin}
\psi(z)-T\psi(z)=-\frac{m M_1}{z^2}+O(\frac{1}{z^3}).
\end{equation}
This, along with~(\ref{rln}), gives us $L=0$. That is,
$\psi$ is a non-zero fixed point of $T$. In turn, it implies that
$m M_1=0$. If $M_1=0$, then $\psi(z)=O(1/z^2)$, and if $M_1\not=0$,
then $m=0$. In either case, Proposition~\ref{rcontr} applies.
It gives a contradiction in this case, too.

($ND_\infty$), i.e. $\sigma=1$ and $m=0$.
One can assume that $k=q-1$, $l=q$. 
Now, assuming that the rank of the matrix is less than $r$, we get that the vectors
$(\frac{\tilde{\partial}^V
\rho_j}{\partial V_1},..., \frac{\tilde{\partial}^V \rho_j}{\partial
V_{q-2}})$, $1\le j\le r$, are linearly dependent.
In the previous notations,
there exists a linear
combination $\psi=\sum_{j=1}^r \beta_j \tilde B_j$, such that
\begin{equation}\label{rln2}
\psi(z)-(T\psi)(z)=\frac{L_{q-1}}{z-V_{q-1}}+\frac{L_{q}}{z-V_{q}},
\end{equation}
where 
$$L_i=\sum_{j=1}^r \beta_j \frac{\tilde{\partial}^V \rho_j}{\partial
V_i}, \ \ i=q-1, q.$$ 
If $M_1$, $M_2$ are defined by
$\psi(z)=\frac{M_1}{z}+\frac{M_2}{z^2}+O(\frac{1}{z^3})$
at infinity, then, from~(\ref{rta}) with $\sigma=1$, $m=0$,
\begin{equation}\label{rpsi-tpsin2}
\psi(z)-T\psi(z)=O(\frac{1}{z^3}).
\end{equation}
This, along with~(\ref{rln2}), gives us two linear relations
$L_{q-1}+L_q=0$, $L_{q-1}V_{q-1}+L_q V_q=0$. 
But $V_{q-1}\not=V_q$. Hence, $L_{q-1}=L_q=0$.
In other words,
$\psi$ is a non-zero fixed point of $T$. 
By Proposition~\ref{rcontr}, it is a contradiction again.
\begin{com}\label{rlast}
This proof demonstrates also the inequalities from Comment~\ref{rfatou}. 
Indeed, otherwise the rows of ${\bf O}$ 
are again linearly dependent, and the proof above applies. 
Observe however, that for this purpose already
the formal identity~(\ref{rruellehyp})
of Theorem~\ref{rformula} (i.e., with some coefficients $L_j$ without
their connections to parameter spaces) would be sufficient.
This approach is somewhat similar to the one
developed first in~\cite{Eps} for
the proof of the Fatou-Shishikura inequality, 
where also more degenerated cases
of the inequality are covered. 
\end{com}

\end{document}